\documentclass[11pt]{amsart}
\usepackage{graphicx}
\usepackage[dvipsnames]{xcolor}
\usepackage{url}
\usepackage[colorlinks=true, pdfstartview=FitV, linkcolor=blue, citecolor=blue, urlcolor=blue]{hyperref}
\usepackage{mathtools}
\usepackage{amscd}
\usepackage{amsmath}
\usepackage{amssymb}
\usepackage[foot]{amsaddr}
\usepackage{amsthm}
\usepackage{cases} 
\usepackage{enumerate}
\usepackage{epstopdf}
\usepackage{ifthen}
\usepackage{enumitem}
\usepackage[margin=3cm]{geometry}
\usepackage{bm}

\newcommand{\CC}{\mathbb{C}}
\newcommand{\C}{\mathbb{C}}
\newcommand{\RR}{\mathbb{R}}
\newcommand{\R}{\mathbb{R}}
\newcommand{\VV}{\mathbb{V}}
\newcommand{\QQ}{\mathbb{Q}}
\newcommand{\calF}{\mathcal{F}}
\newcommand{\calH}{\mathcal{H}}
\newcommand{\calL}{\mathcal{L}}

\newcommand{\trace}{\mathrm{trace}}
\newcommand{\Rank}{\mathrm{rank}}
\newcommand{\Sign}{\mathrm{sign}}
\newtheorem{theorem}{Theorem}[section]

\newtheorem{proposition}[theorem]{Proposition}

\theoremstyle{definition}
\newtheorem{definition}[theorem]{Definition}
\newtheorem{example}[theorem]{Example}
\theoremstyle{remark}
\newtheorem{remark}[theorem]{Remark}
\numberwithin{equation}{section}
\usepackage[ruled,vlined]{algorithm2e}
\SetKwComment{Comment}{}{}

\begin{document}
\title[Time-optimal neural feedback control as binary classification]{Time-optimal neural feedback control of nilpotent systems as a binary classification problem}
\author[Bicego]{Sara Bicego}
\email[Bicego]{sara.bicego@outlook.com}
\author[Gue]{Samuel Gue}
\email[Gue]{s.c.m.gue.2235336@swansea.ac.uk}
\author[Kalise]{Dante Kalise}
\email[Kalise]{dkaliseb@ic.ac.uk}
\author[Villamizar]{Nelly Villamizar}
\email[Villamizar]{n.y.villamizar@swansea.ac.uk }
\address[Bicego, Kalise]{Department of Mathematics, Imperial College London, United Kingdom}
\address[Gue, Villamizar]{Department of Mathematics, Swansea University, United Kingdom}
\date{\today}

\begin{abstract}
A computational method for the synthesis of time-optimal feedback control laws for linear nilpotent systems is proposed. The method is based on the use of the bang-bang theorem, which leads to a characterization of the time-optimal trajectory as a parameter-dependent polynomial system for the control switching sequence. A deflated Newton’s method is then applied to exhaust all the real roots of the polynomial system. The root-finding procedure is informed by the Hermite quadratic form, which provides a sharp estimate on the number of real roots to be found. In the second part of the paper, the polynomial systems are sampled and solved to generate a synthetic dataset for the construction of a time-optimal deep neural network -interpreted as a binary classifier- via supervised learning. Numerical tests in integrators of increasing dimension assess the accuracy, robustness, and real-time-control capabilities of the approximate control law.
\end{abstract}
\maketitle

\section{Introduction}
Time-optimal control problems have a prominent place in control theory due to their many applications in robotics, aerospace, and trajectory planning \cite{TOCP,TOCR,TOCC}. A classical result dating back to the early days of optimal control, known as the \textsl{bang-bang} theorem states that, for finite-dimensional, linear time-optimal processes where the control variable is constrained to a convex, closed, and bounded polyhedron, the optimal control signal is piecewise constant, taking all its values at the vertices of the polyhedron (see the different formulations  in e.g. \cite{BellmanBB,LaSalle,PMP}). Moreover, under additional properties of the state-to-state map, an upper bound on the number of switches of the piecewise-constant control can be determined. Both statements will be made more precise in the next section. From a state space perspective, such a result transforms the time-optimal control synthesis into a geometric problem, that is, the identification of a switching surface splitting the state space into regions where different vertices of the control polyhedron are selected as the instantaneous optimal control action. Hence, while the original bang-bang theorem follows from first-order optimality conditions for a given initial state, it also provides a characterization of the optimal feedback control law as a piecewise constant map with a finite number of output values. With no loss of generality, in this paper we restrict our presentation to scalar control signals taking values in $[-1,1]$, for which the optimal feedback law corresponds to a binary classifier to $\{-1,1\}$. We study the construction of such a binary classifier via supervised learning with synthetic data obtained from sampling optimal trajectories, which are parametrized in terms of the switching structure of the optimal control signal.

In the first part of this paper, we follow a similar approach as in \cite{nilpotentFeedback}. Here, the authors cast a time-optimal control problem, for which a parametrization of the switching control sequence and time integration of the dynamics leads to a polynomial system of equations for the switching times. Since there is a unique solution to the time-optimal control problem, we only need to determine the existence of a non-negative solution of the polynomial system given a choice of the initial term of the switching control sequence, -1 or 1. This solvability problem is studied using Gr\"obner bases, Sturm sequences, and \texttt{Macaulay2} \cite{M2} as an effective computer algebra system for related calculations. Extensions of this work to a more general class of linear dynamical systems have been studied in \cite{patil,PATIL20151,RauscherSawodnyChainIntegrators}. A comprehensive review of the use of Gr\"obner bases in control theory is given in \cite{georgbook}. However, to the best of our knowledge, the computational complexity associated with the use of Gr\"obner bases for solving polynomial systems such as those arising in time-optimal control hinders their applicability to higher-dimensional systems than those reported in the aforementioned references and makes them unsuitable for real-time feedback control. Our contribution is to propose a methodology that circumvents both issues:
\begin{itemize}[itemsep=0.1cm,topsep=0.1cm]
 \renewcommand{\labelitemi}{\tiny{$\blacksquare$}}
\item Instead of analyzing the existence of non-negative solutions to the polynomial system, we apply a variant of Newton's method which can effectively exhaust all possible roots. We follow a deflation technique as originally proposed in \cite{brown70} for polynomial systems and later extended to a wider class of problems and applications \cite{bicego2024,deflation_2,deflation,deflation_3}. 
\item The stopping of the deflation method is determined by the number of roots it is expected to find. We apply Hermite's quadratic form as a mechanism to establish the number of roots of the polynomial systems under study. The applications of quadratic forms in root counting has been explored from as early as 1853 by Sylvester in \cite{SylvesterQuadraticForms}. The idea of using Hermite quadratic forms to solve polynomial systems with real parameters is discussed in \cite{PhuocLeSafeyElDinHermiteMatrices} and its use for sign determination in \cite{AlgorithmsBasuPollackRoy,BKR-SignDetermination,CohenSignDetermination}. 
A recent application of the above to semi-algebraic systems with real parameters can be found in \cite{GSeDRootCounting}.

\item A real-time feedback control is built using supervised learning. The block associated to the solution of the polynomial system receives as an input an initial condition, which enters as a parameter in the polynomial system, and outputs the first control value of the optimal switching sequence, which leads to the construction of a binary classifier mapping the $n$-th dimensional state space to $\{-1,1\}$. Inspired by recent work on supervised learning for optimal feedback control \cite{MPC_bookchapter,AKK,Cipriani,DKS,Lucia,NN_MPC,Oster2024} we train a deep neural network to approximate the state-to-control feedback map using synthetic data. This approach bypasses the need for iterative recomputation of optimal trajectories, which is replaced by real-time model evaluation. 
\end{itemize}
The rest of the paper is organized as follows. In Section 2, we introduce nilpotent linear time-invariant systems with a chain of integrators as a prototypical case, and we provide a characterization of the time-optimal control problem as a polynomial system for the switching times of the control sequence. In Section 3, we present the deflated Newton's method we apply to find all possible roots of the polynomial system. In Section 4, we introduce the Hermite quadratic form and its use to determine the existence of non-negative real roots of a system of polynomial equations. In Section 5, we discuss the approximation of a binary classifier via supervised learning with deep neural networks and synthetic data generated from sampling polynomial systems for different initial conditions. In Section 6, we assess the performance and robustness of the proposed approach on chains of integrators up to dimension 5. Section 7 presents concluding remarks and future research directions.

 \section{Time-optimal control of nilpotent systems}
We study the linear time-optimal control problem
\begin{align}\label{eq:linearSystem}
 \underset{u\in\mathcal{U}}{\min} \;\;T\quad 
 \text{subject to} \quad
 &\begin{cases}\dot{\bm{y}}(t)=A\bm{y}(t)+\bm{b}u(t)\\
 \bm{y}(0)=\bm{x}\\
 \bm{y}(T)=\bm{0},
 \end{cases}
\end{align}
where $\mathcal{U}\equiv L^{\infty}\bigl([0,+\infty);[-1,1]\bigr)$, $A$ is an $n\times n$ real matrix, $\bm{b}\in\RR^n$, and $\bm{x}\in\RR^n$ a given initial condition. We begin by recalling a well-known characterization of the solution to time-optimal control problem \eqref{eq:linearSystem}.

\begin{theorem}[Chapter III, \cite{PMP}]\label{switching} Assume that Kalman's controllability condition holds for the pair (A,$\bm{b}$), and that all the eigenvalues of $A$ are real. Then, for every initial condition $\bm{x}$, there exists a unique time-optimal control signal $u$. This control signal is bang-bang, i.e. taking values in $\{-1,1\}$, with at most $n-1$ switches. 
\end{theorem}

The result above motivates a parametrization of the total time $T=t_1+\cdots+t_n$ for non-negative $t_i$'s
, which determine the switching times of the control signal. 
 The time domain $\Omega$ is the union of subintervals $\Omega_i=[T_{i-1},T_i)$, with $T_0=0$, and $T_i=\sum_{j=1}^i t_j$, for every $1\leq i\leq n-1$, and $\Omega_n=[T_{n-1},T_n]$. This induces a further parametrization of the optimal control signal $u$ as a piecewise-constant function in time, taking the value $u(0)=\pm 1$ in $\Omega_1$, and $(-1)^{i-1}u(0)$ in the interval $\Omega_i$. When a control has fewer than $n-1$ switches, this coincides with $t_i=0$ for some values of $i=1,\ldots,n$. When $t_i=0$, we define $\Omega_i=\emptyset$. We study the case where the matrix $A$ is a nilpotent matrix with nilpotency index $n$. Here, $A=(A_{ij})_{i,j}$ is similar to an $n\times n$ nilpotent Jordan block and hence, without loss of generality, we assume that $A_{ij}=\delta_{i,j-1}$, for $i,j=1, \dots, n$, where $\delta$ is the Kronecker delta. This covers a class of relevant problems, such as $n$-th order integrators of the form
\begin{align*}
\dot y_1(t)&=y_2(t)\\
&\vdots\\
\dot y_{n-1}(t)&=y_n(t)\\
\dot y_n(t)&=u(t)\,,
\end{align*}
and allows the following characterization of the optimal control signal.
\begin{proposition}\label{prop:21}
 Let $A$ be an $n\times n$ nilpotent Jordan block and $\bm{b}\in\R^n$ satisfying the assumptions in Theorem \ref{switching}. Given an initial condition $\bm{x}$, the non-negative increments $\{t_i\}_{i=1}^n$ such that $T=t_1+\cdots+t_n$, associated to the time-optimal control signal $u(t)=u_i$, for all $t \in \Omega_i$, with $u_i=\pm 1$, satisfy the polynomial system
 \begin{equation}\label{eq:polynomialSystem2}
 0=-u(0){\sum_{k=1}^{n-i+1} b_{k+i-1}}\sum_{\substack{\alpha_1,\ldots,\alpha_n\geq0\\
 \alpha_1+\cdots+\alpha_n=k}}(-1)^{\max \{j\colon \alpha_j\neq0\}}\frac{t_1^{\alpha_1}\cdots t_n^{\alpha_n}}{\alpha_1!\cdots\alpha_n!}+{\sum_{k=i}^{n}\frac{T^{k-i}}{(k-i)!}x_{k}},
 \end{equation}
 for $i=1,\dots,n,$ where $b_j, x_j$ denote the coordinates of $\bm b$ and $\bm x$, respectively.
 \end{proposition} 
 \begin{proof}
 First, recall that for a given control signal, the trajectory associated to the linear system in \eqref{eq:linearSystem} is given by 
 \begin{equation*}
 \bm{y}(t)=e^{tA}\bm{y}(0)+e^{tA}\int_{0}^{t}e^{-\tau A}\bm{b}\,u(\tau)\,d\tau. 
 \end{equation*} 
Evaluating at $T=t_1+\cdots+t_n$, and imposing that $\bm{y}(T)=\bm{0}$ implies
 \[\bm{0}=e^{TA}\bm{x}+e^{TA}\int_{0}^{T}e^{-\tau A}\bm{b}\,u(\tau)\, d\tau,
 \]
 and integration by parts yields
 \begin{align}
 \bm{0}&=e^{AT}\bm{x}+e^{TA}\sum_{k=1}^{n} \bigl(A^{k-1}\, e^{-T A}\bm b\, U_k(T)-A^{k-1}\bm b\,U_k(0)\bigr)=e^{TA}\bm{x}+\sum_{k=1}^{n} A^{k-1}\,\bm{b}\,U_k(T),\label{eq:eq}
 \end{align}
 where 
 \[U_{k}(t)=\int_0^t U_{k-1}(\tau)\, d\tau, \text{ for\; } i\geq 1, \text{ with\; } U_1(t)=\int_{0}^t u(\tau) d\tau.
 \]
 Since $u$ is piecewise constant on the partition of $\Omega$ determined by the $t_i$'s, 
 \[U_1(T)=u(0)\sum_{i=1}^{n}\int_{T_{i-1}}^{T_i} (-1)^{i-1}d\tau=u(0)(t_1-t_2+\dots(-1)^{n-1}t_n),\] 
 where, as above, $T_i=t_1+\cdots+t_i$. 
 For any $1\leq k\leq n-1$, we have
 \begin{align*} 
 U_{k}(T)&=\int_{0}^{T}U_{k-1}(\tau)d\tau=\sum_{i=1}^{n}\int_{T_{i-1}}^{T_i}U_{k-1}(\tau)d\tau
 =
 u(0)\sum_{i=1}^n\int_{T_{i-1}}^{T_i} (-1)^{i-1}\frac{\tau^{k-1}}{(k-1)!}d\tau
 \\
 &= u(0)\sum_{i=1}^n (-1)^{i-1}\biggl[\frac{(t_i+T_{i-1})^k}{k!}-\frac{(T_{i-1})^k}{k!}\biggr]= u(0)\sum_{i=1}^n\sum_{j=1}^k\frac{(-1)^{i-1}}{(k-j)!j!}t_i^jT_{i-1}^{k-j}.\nonumber
 \end{align*}
 Since
 \[\left(t_1+\cdots+t_{i-1}\right)^{k-j}
 =
 \sum_{\substack{\alpha_1+\cdots+\alpha_{i-1}=k-j\\ \alpha_{1},\dots,\alpha_{i-1}\geq 0}}\frac{(k-j)!}{\alpha_{1}! \alpha_{2}!\cdots \alpha_{i-1}!}t_1^{\alpha_1} \cdots t_{i-1}^{\alpha_{i-1}},
 \]
 then
 \begin{align}
 U_{k}(T)&=
 u(0)\sum_{i=1}^n\sum_{\substack{\alpha_1+\cdots+\alpha_{i}=k\\\alpha_i,\dots,\alpha_{i-1}\geq 0,\,\alpha_i>0}}\frac{(-1)^{i-1}}{\alpha_{1}!\cdots \alpha_{i}!}t_1^{\alpha_1} \cdots t_{i}^{\alpha_{i}}\nonumber\\
 &=
 -u(0)\sum_{\substack{\alpha_1+\dots+\alpha_n=k\\\alpha_1,\ldots,\alpha_n\geq0}}\frac{(-1)^{\max\{j\colon \alpha_j\neq 0\}}}{\alpha_{1}! \cdots \alpha_{n}!}t_1^{\alpha_1}\cdots t_n^{\alpha_n}.\label{eq:sum}
 \end{align}
 On the other hand, since $A$ is a nilpotent Jordan block, we have $A_{ij}=\delta_{i,j-1}$, for $i,j=1, \dots, n$.
 Then, for $1\leq k\leq n$ we have $(A^{k-1})_{i,j}=\delta_{i,i+k-1}$. Thus, $\bigl( A^{k-1}\bm{b}\bigr)_{i}=b_{k+i-1}$ for $i=1,\dots, n-k+1$, and $\bigl( A^{k-1}\bm{b}\bigr)_{i}=0$ otherwise.
 
 Finally, replacing \eqref{eq:sum} in \eqref{eq:eq}, and noting that the $i$-th component of $e^{TA}\bm{x}$ is given by
 \begin{equation*}
 (e^{TA}\bm{x})_i=\sum_{k=1}^{n}(e^{TA})_{i,k}{x}_k=\sum_{k=i}^{n}\frac{T^{k-i}}{(k-i)!}{x}_k,
 \end{equation*} 
 we obtain the equality \eqref{eq:polynomialSystem2}, as required.
 \end{proof}
\begin{remark}
Note that the control signal $u(t)$ need not be optimal for the proof 
of Proposition~\ref{prop:21} to hold, as it only 
requires that $u(t)$ is parametrized by the non-negative increments 
$\{t_i\}_{i=1}^{n}$ and $u(0)$. The connection to optimality is 
provided by Theorem~\ref{switching}, which guarantees that the 
time-optimal control is bang-bang with at most $n-1$ switches and 
hence admits such a parametrization. Moreover, for the origin as 
target point, the sets of states reachable with $u(0)=1$ and $u(0)=-1$ 
are disjoint up to a set of measure zero \cite{patil}, so that for 
almost every $x$, there exists a unique $u(0)\in\{-1,+1\}$ for 
which~\eqref{eq:polynomialSystem2} admits a non-negative solution. 
Among all non-negative solutions, the time-optimal one minimizes 
$T = t_1 + \cdots + t_n$.
\end{remark}
 \begin{example}\label{ex:1}
 In $\R^3$, consider the vector $\bm{b}=(0,0,1)^\top $. A solution of \eqref{eq:linearSystem} leads to the system of polynomial equations: 
 \begin{align*}
 0&=u(0)(t_1-t_2+t_3)+x_3,\\
 0&=u(0)\left(\frac{t_1^2}{2}+t_1t_2+t_1t_3-\frac{t_2^2}{2}-t_2t_3+\frac{t_3^2}{2}\right)+x_2+(t_1+t_2+t_3)x_3,\\
 0&=u(0)\left(\frac{t_3^3}{6}+\frac{t_1^2t_2}{2}+\frac{t_1^2t_3}{2}+\frac{t_1t_2^2}{2}+t_1t_2t_3+\frac{t_1t_3^2}{2}-\frac{t_2^3}{6}-\frac{t_2^2t_3}{2}-\frac{t_2t_3^2}{2}+\frac{t_3^3}{6}\right)\\
 &\;\;\;\;\;\;+\;x_1+(t_1+t_2+t_3)x_2+\frac{(t_1+t_2+t_3)^2}{2}x_3,
 \end{align*}
 and $u(0)=\pm 1$.\hfill$\diamond$
 \end{example}

With no loss of generality\footnote{it is well-known that for scalar control systems in Jordan form, controllability is equivalent to requiring $b_n\neq 0$ \cite{CD68}.}, we will henceforth assume $\bm{b} = (0, \dots, 0, 1)^\top$, leading to a simplified version of \eqref{eq:polynomialSystem2} 
\begin{equation}\label{eq:polynomialSystemB0}
0=-u(0)\sum_{\substack{\alpha_1,\ldots,\alpha_n\geq0\\
\alpha_1+\cdots+\alpha_n=n-i+1}}(-1)^{\max \{j\colon \alpha_j\neq0\}}\frac{t_1^{\alpha_1}\cdots t_n^{\alpha_n}}{\alpha_1!\cdots\alpha_n!}+{\sum_{k=i}^{n}\frac{(t_1+\cdots+t_n)^{k-i}}{(k-i)!}x_{k}},
\end{equation}
 for $i=1,\ldots,n$. This system characterizes the optimal control signal in the following way. For every initial condition $\bm{x}$, there exists a unique optimal control signal, which is determined by the choice of the initial control $u(0)=\pm 1$, and the existence of a set of non-negative $t_i$'s related to that choice. In other words, if a non-negative sequence of $t_i$'s is found for a given $\bm{x}$ and $u(0)$, then the switching sequence is the time-optimal solution to the problem. Otherwise, we are guaranteed a unique non-negative solution to \eqref{eq:polynomialSystemB0} by taking $-u(0)$ as the starting element of the sequence and computing the associated switching times. In the following section, we turn our attention to Newton's method to find roots of the polynomial system.

\section{A deflated Newton's Method}\label{newt}
 
In this section, we present a numerical technique that combines Newton's method with a deflation routine to determine all possible solutions of the polynomial system \eqref{eq:polynomialSystemB0}. 

We denote by $f_1,\dots,f_n$ the polynomials in $\bm t=(t_1,\dots, t_n)$ from \eqref{eq:polynomialSystemB0}, which reads
\begin{equation}\label{F}
\calF (\bm t)= \bigl(
 f_1 (\bm t),
 \dots,
 f_n (\bm t)\bigr)^\top=\bm{0}.
 \end{equation}
Given an initial guess $\bm{t}_0$, Newton's method generates a sequence 
\begin{align}\label{nsys}
\bm{t}_{h+1}=\bm{t}_{h}-\calF'(\bm{t}_{h})^{-1}\calF(\bm{t}_{h})
\end{align}
where $\mathcal{F}'$ denotes the Jacobian of $\mathcal{F}$. Assuming that this iteration converges to a root of \eqref{eq:polynomialSystemB0}, we have no guarantees that it will converge to a non-negative solution and, unless all solutions have been found, this cannot be interpreted as non-existence of non-negative solutions. Hence, we resort to a deflated Newton method to find all possible roots, which we describe in the following.

When using an iterative algorithm to find the roots of a function $\mathcal{F}(\bm{t})$, one can search for new, distinct roots of $\mathcal{F}(\bm{t})$ by applying the same iterative algorithm to a new deflated function $\mathcal{G}( {\bm t})$ whose zeros coincide with the roots of $\mathcal{F}( {\bm t})$, except for those already found.
Specifically, if $\mathcal{F}(\bm{r})=\bm{0}$, for some $\bm r\in\R^n$, we define
\begin{equation*}
\mathcal{G}(\bm{t}) =\mathcal{M}(\bm{t}, \bm{r})\mathcal{F}(\bm{t}),
\end{equation*}
where $\mathcal{M}(\bm{t}, \bm{r} )$ is the \emph{deflation matrix}, which is invertible for all $\bm{t}\in \mathbb{R}^n\setminus\{\bm{r}\}$, and $\mathcal{M}(\bm{t}, \bm{r} )\mathcal{F}(\bm{t}) = \bm{0}$ if and only if $\mathcal{F}(\bm{t})=\bm{0}$. 
This involves systematically altering the residual of the polynomial system to eliminate solutions that have already been identified. In this work, we consider 
\begin{equation}\label{defl}
\mathcal{G}(\bm{t}) = \bigg(\mathbb{I}\frac{1}{\eta(\bm{t})} + \xi\mathbb{I}\bigg)\mathcal{F}(\bm{t})\,,\; \text{for } \eta(\bm{t}) = \|\bm{t} - \bm{r} \|_2^p\,,
\end{equation}
 where $\mathbb{I}\in\mathbb{R}^{n\times n}$ is the identity matrix, {the \emph{deflation power} $p$ is a positive integer, and the \emph{deflated residual} $\xi$ is a non-negative real number. }
 The \emph{deflating term} $\eta(\bm{t})^{-1}$ ensures that $\mathcal{G}( \bm{r} )\neq \bm{0}$, while the shift $\xi\geq0$ forces the deflated residual to not vanish artificially as $ \|\bm{t} - \bm{r} \|_2\to\infty$. 
 
 The use of Newton's method requires the computation of the derivative of $\mathcal{G}(\bm{t})$. This can be done efficiently both in memory and computational cost by leveraging the information about the previous deflation iterates. 
 If $\bm{r}_0,\dots,\bm{r}_{m-1}$ are roots of $\mathcal{F}$, at the {$m$-th} iteration, the deflation term reads
 \begin{equation*}
 \eta_m = \eta_{m-1}(\bm{t})\tau(\bm{t})\,,\; \text{ with\; }
 \eta_{m-1}(\bm{t}) = \prod\limits_{i=0}^{m-1}\|\bm{t}- \bm{r}_i \|_2^p\,,
 \end{equation*}
 for a temporal variable $\tau(\bm{t}) = \|\bm{t}- \bm{r}_m \|_2^p$.
 The derivative of the polynomial system after $m$ deflations is given by
 \begin{equation}\label{d_defl}
 \mathcal{G}' = \dfrac{\mathcal{F}'}{\eta_{m}} - \dfrac{\mathcal{F}}{\eta_{m}^2}\otimes\big(\eta_{m-1}'\,\tau + \eta_{m-1}\,\tau'\big)\,.
 \end{equation} 
For the time-optimal polynomial system \eqref{eq:polynomialSystemB0}, the partial derivative of the $i$-th component $f_i$ of $\calF$ with respect to the switching time $t_q$ is given by the polynomial
\begin{equation*}
\partial_{t_q}f_i=-u(0)\sum_{\substack{\alpha_1,\ldots,\alpha_n\geq0, \, \alpha_q\neq 0\\\alpha_1+\cdots+\alpha_n=n-i+1}}(-1)^{\max\{j\colon\alpha_j\neq0\}}\frac{t_1^{\alpha_1}\cdots t_q^{\alpha_q-1}\cdots t_n^{\alpha_n}}{\alpha_1!\cdots(\alpha_q-1)!\cdots\alpha_n!}+
\sum_{k=i+1}^{n}\frac{T^{k-i-1}}{(k-i-1)!}x_k.
\end{equation*}
 A pseudocode for this routine is provided in Algorithm \ref{alg:deflation}. A crucial consideration for the deflation algorithm is the selection of suitable stopping criteria, as once all roots have been found, the algorithm will naturally fail to converge. However, numerical instability can also prevent convergence as the number of roots in the deflation operator increases. Therefore, establishing a bound on the number of real solutions in our polynomial system is fundamental. The next section explores this bound by drawing from real algebraic geometry and utilizing the Hermite quadratic form.

 \begin{algorithm}[ht!]
 \caption{Deflated Newton}
 \label{alg:deflation}
 \KwData{Function $\mathcal F$ and its derivative $\mathcal F'$, initial $\texttt{guess}$, expected number of roots $n_{\texttt{roots}}$, deflation parameters $p$, $\xi$, convergence tolerance $\varepsilon\ll1$}
 
 $\texttt{max\_iteration}=10^3$, $\texttt{solutions} \gets \{\}$\Comment*[r]{initializing}
 $\eta_0(*) \gets 1$, $ \eta'_0(*) \gets 0$, $\tau(*) \gets 1$, $ \tau'(*) \gets 0$ \Comment*[r]{deflation operators}
 $m \gets 0$ \Comment*[r]{count of identified roots}
 \While{$m<n_{\texttt{roots}}$}{
 $\bm{t} \gets \texttt{guess}$ \Comment*[r]{initial guess}
 $\texttt{update}\gg1$, $\texttt{iteration} \gets 0$\Comment*[r]{reset convergence flags}
 \While{$\texttt{update} > \varepsilon$ and $
 \texttt{iteration} < \texttt{max\_iteration}$}{
 $\texttt{iteration}\, += \,1$;\\
 $f \gets \mathcal F(\bm{t})$, $df \gets \mathcal F'(\bm{t})$;\\
 $\mathcal G \gets \eqref{defl},\;\mathcal G' \gets \eqref{d_defl}$\Comment*[r]{compute deflated system}
 $\bm{t}_{\text{new}} \gets \bm{t} - \texttt{pinv}(\mathcal G')\cdot \mathcal G$\Comment*[r]{update root candidate}
 $\texttt{update} \gets \|\bm{t}_{\text{new}} - \bm{t}\|$;\\
 $\bm{t} \gets \bm{t}_{\text{new}}$;\\
}
 \If{$\texttt{iteration} = \texttt{max\_iteration}$}{
 $\textbf{break}$}
 \Else{
 $\texttt{solutions} \gets \texttt{solutions} \cup \{\bm{t}\}$ \Comment*[r]{store solution }
 $m\, += \,1$\\
 $\tau(*) \gets \|* - \bm{t}\|^p$ \Comment*[r]{Update deflation operators}
 $ \tau'(*) \gets p \cdot \|* - \bm{t}\|^{p-2} \cdot (* - \bm{t})$;\\
 $\eta_m(*) = \eta_{m-1}(*)\tau(*)$;\\
 $\eta'_{m}(*) = \eta'_{m-1}(*)\tau(*) +  \eta_{m-1}(*)\tau'(*)$;
 }}
 \KwResult{list $\texttt{solutions}$, containing the roots of $\mathcal{F}$ }
 \end{algorithm}

 \section{The Hermite Quadratic Form}\label{hermy}
 
We present the Hermite quadratic form method, as introduced in \cite{AlgorithmsBasuPollackRoy}, as a tool to determine the number of roots in the polynomial system \eqref{eq:polynomialSystemB0}. 
We denote by $R=\R[t_1,\dots,t_n]$ the ring of polynomials in $\bm{t}$ with real coefficients. 
Let 
\begin{equation}\label{eq:ideal}
 J=\langle f_i\colon i=1,\dots, n\rangle=\biggl\{\sum_{i=1}^ng_if_i\colon g_i\in R\biggr\}\subseteq R
\end{equation} 
be the ideal of $R$ generated by the polynomials $f_i\in R$ in \eqref{F}.
The ideal $J$ is said to be \emph{zero-dimensional} if the set
$R/J=\{g+J\colon g\in R\}$ 
of polynomials module $J$ forms a finite dimensional real vector space, where $g+J=\{g+f\in R\colon f\in J\}$, for any $g\in R$. 

To apply the Hermite quadratic form method, it is necessary to ensure that $R/J$ is a finite-dimensional vector space, which is equivalent to $J$ being zero-dimensional.
This property can be verified using a monomial ordering, which can be used either to establish certain properties of the monomials which are powers of the variables $t_i$ or to compute a Gröbner basis for the ideal $J$.
Alternatively, if we can verify that 
the polynomials $f_i$ generating $J$ have only a finite number of solutions in $\C$, this also implies that $J$ is zero-dimensional.
For further details on these methods, see, for example, \cite[§3, Chapter 5]{idealsVarietiesAlgorithms}. 
\begin{definition}[Hermite quadratic form]
If a polynomial system $\bigl\{f_i\bigr\}_{i=1}^n$ 
in $R$ has a finite number of solutions in $\C$, and $\mathcal{B}={\{\beta_1,\dots,\beta_r\}}$ is a basis of $R/J$, for $J$ as in \eqref{eq:ideal}, let $\mathcal{L}_{ij}(J)$ be the map on $R/J$ defined by
\begin{align*}
f+J&\mapsto {\beta_i\beta_j}f+J,
\end{align*}
for $1\leq i,j\leq r$. The \emph{Hermite quadratic form} of $J$ is the map 
on $R/J$ defined by 
the matrix $(\mathcal{H}(J))_{ij}=\trace(L_{ij}(J))$, where $L_{ij}(J)$ is the matrix of the linear map $\mathcal{L}_{ij}(J)$. In particular, $\calH(J)$ is a real symmetric matrix. 
\end{definition}
\begin{example}\label{exampleH}
Taking $n=2$ and $u(0)=1$ in \eqref{eq:polynomialSystemB0} leads to
\begin{align}
 f_1&=t_1-t_2+x_2,\label{eq:1}\\
 f_2&=\frac{t_1^2}{2}+t_1t_2-\frac{t_2^2}{2}+x_1+(t_1+t_2)x_2.\label{eq:2}
\end{align}
These are polynomials in $R=\R[t_1,t_2]$, and we take $J=\langle f_1,f_2\rangle$. This system has a finite number of solutions because if $ t_2=t_1+x_2$ by \eqref{eq:1}, then \eqref{eq:2} leads to a quadratic equation in $t_1$, which has at most two solutions. 
We see that $\mathcal{B} = \{1,t_2\}$ is a basis for the vector space $R/J$. Taking $\beta_1=1$, and $\beta_2=t_2$, then the map $\calL_{12}(J)$ is given by $1+J\mapsto t_2+J$, and $t_2+J\mapsto t_2^2+J= -x_1+\frac{1}{2}x_2^2 +J$. Thus, 
\[
L_{12}(J)=
\begin{pmatrix}
 0&-x_1+\frac{1}{2}x_2^2\\
 1&0
 \end{pmatrix},
\]
and so $\trace(L_{12}(J))=0$. 
Similarly, we compute $L_{ij}(J)$ for $i,j=1,2$, and we get the Hermite quadratic form of $J$, which is given by the matrix
$\mathcal{H}(J)=
\begin{pmatrix}
 2&0\\
 0&x_2^2-2x_1
\end{pmatrix}.$\hfill$\diamond$
\end{example}
In practice, for a given value of $\bm x$, we use \texttt{Macaulay2} \cite{M2} to compute the dimension of the ideal $J$ and verify if it is zero. The Hermite quadratic form for parametric systems of polynomial equations has been studied in \cite{PhuocLeSafeyElDinHermiteMatrices}. Here, it is shown that for a given parameter, it is possible to obtain the Hermite quadratic of the associated polynomial system, as long as the denominator of the parameter-augmented Hermite quadratic form does not vanish.

If the ideal $J$ generated by $\bigl\{f_i\bigr\}_{i=1}^n$ in \eqref{F}
is zero-dimensional, the Hermite quadratic form $\mathcal{H}(J)$ can be used to count the number of real roots of $\calF$ as follows. 
First, recall that the rank of the matrix $\calH(J)$, denoted $\Rank(\calH(J))$, is equal to the number of non-zero eigenvalues of $\calH(J)$ counted with multiplicity, and the \emph{signature}, denoted $\Sign(\calH(J))$, is the difference between the number of its positive and negative eigenvalues. 
Additionally, we denote by $\VV_{\RR}(J)$ and $\VV_{\CC}(J)$ the set of real and complex solutions of $\calF$, respectively. The following result, which is a special case of the one presented in \cite[Theorem 4.100]{AlgorithmsBasuPollackRoy}, establishes a link between the Hermite quadratic form $\calH(J)$ and the number of roots of $\calF$.
\begin{proposition}\label{RootCountResult}
Suppose the system of polynomial equations $\bigl\{f_i\bigr\}_{i=1}^n$ in $R$ has finitely many solutions. 
Let $J=\langle f_i\colon i=1,\dots, n\rangle$, and $\mathcal{H}(J)$ be the Hermite quadratic form of $J$. Then, 
\[ \Rank(\mathcal{H}(J))=\#\VV_\CC(J),\; \text{and\; }
 \Sign(\mathcal{H}(J))=\#\VV_\RR(J).
\]
\end{proposition}
\begin{remark}
By Proposition \ref{RootCountResult}, the number of real roots of a polynomial system can be computed by counting the positive and negative real roots of the characteristic polynomial of the matrix $\mathcal{H}(J)$.
\end{remark}
\begin{example}
In Example \ref{exampleH}, the signature of $\mathcal{H}(J)$ is determined by the sign of the eigenvalue $x_2^2-2x_1$. 
If $x_2^2-2x_1>0$, then both eigenvalues of $\mathcal{H}(J)$ are positive, meaning that the system of polynomial equations has two real solutions. If $x_2^2-2x_1=0$, then the system has only one real solution. If $x_2^2-2x_1<0$, then $\Sign(\mathcal{H}(J))=0$, and the system has no real solutions. These three cases are illustrated in Figure \ref{fig:sign}\hfill$\diamond$
\begin{figure}[h]\label{geometryExample}
 \centering 
 \includegraphics[scale=0.34]{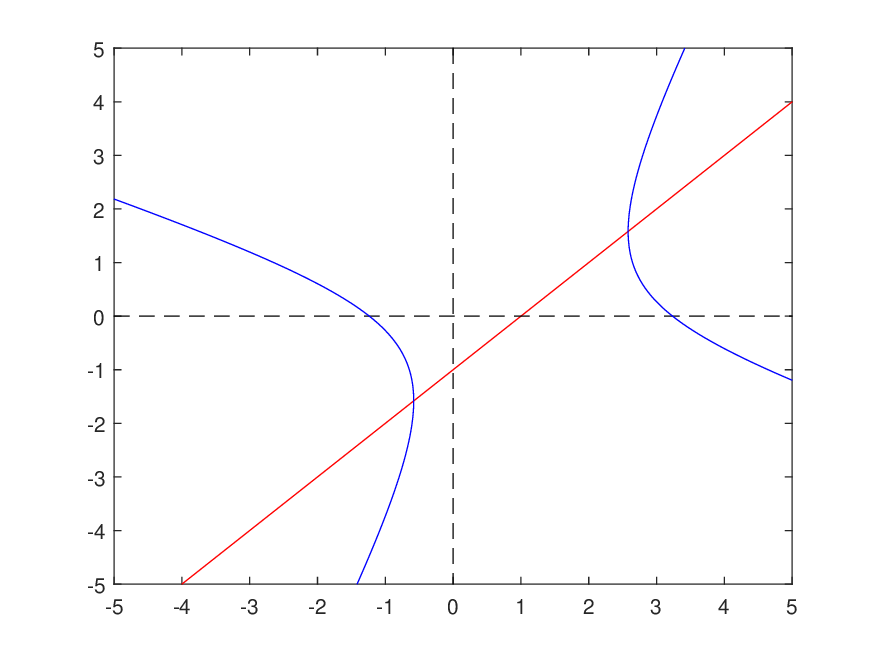}
 \includegraphics[scale=0.34]{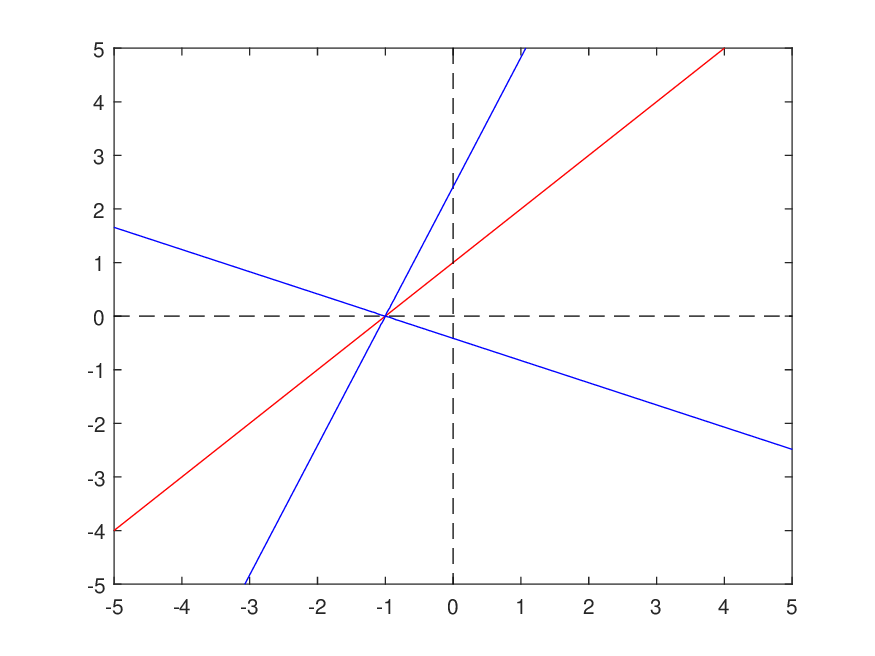}
 \includegraphics[scale=0.34]{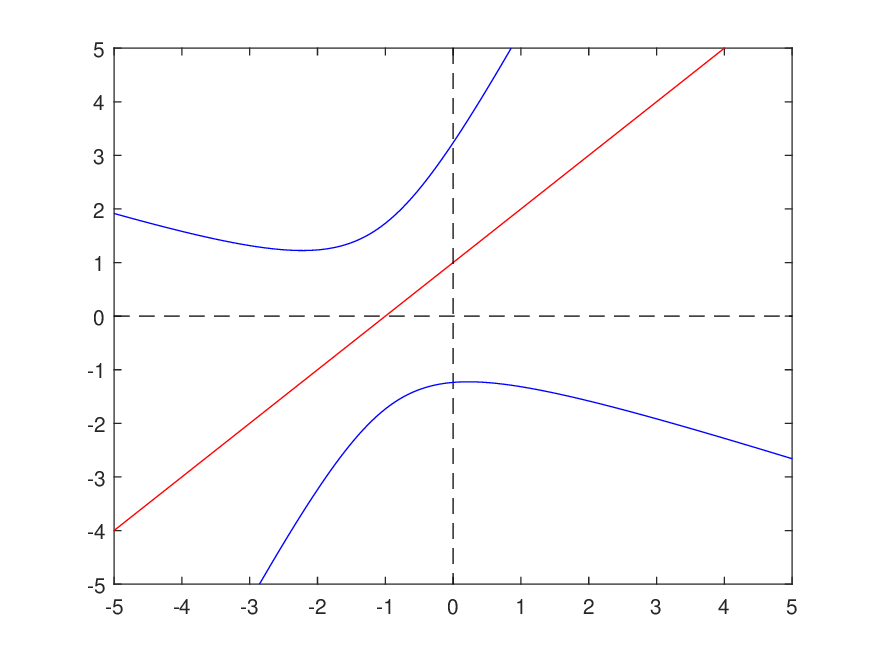}
 \caption{{Graph of the two dimensional polynomial system from Example \ref{exampleH}. The number of real solutions of the system depends on the sign of the diagonal elements in the matrix $\calH(J)$. We have three cases: $x_2^2-2x_1>0$ (left), $x_2^2-2x_1=0$ (middle), and $x_2^2-2x_1<0$ (right).}}\label{fig:sign}
\end{figure}
\end{example}
The Hermite quadratic form feeds into Algorithm \ref{alg:deflation} by providing $n_\text{roots}$, the expected number of real roots. We set $n_\text{roots}=\Sign(\mathcal{H}(J))$. This quantity is valid for fixed $\bm{x}$ and $u(0)$, and once $n_\text{roots}$ has been reached without finding a non-negative solution, we are guaranteed the existence of a non-negative solution by taking $-u(0)$.

All of our computations are completed in \texttt{Macaulay2} using the \texttt{RealRoots} package \cite{M2RealRoots}. The command ``$\mathrm{traceCount}(J)$" returns the signature of Hermite quadratic form of $J$. Within this computation, a Gröbner basis of $J$ is computed when finding a basis of $R/J$. As shown in \cite{GroebnerComplexity}, the algorithms to compute the Gröbner basis of the ideal $J$ will generate polynomials with total degree bounded above by $2(n^2/2+n)^{2^{n-1}}$. As a result, as we increase the number of variables $n$, the Gröbner basis computation will increase double-exponentially. In Section 6, we report results where the Hermite quadratic form has been computed for polynomial systems up to dimension 5, although computations have been successfully performed up to dimension 7. 

\section{Constructing a neural feedback law as a binary classifier}
The methodology developed so far generates open-loop controls, which are effective in an ideal setting where both system dynamics and initial conditions are exactly known. However, in real-world applications, model uncertainties and disturbances will lead to deviations from the optimal trajectory, requiring a re-computation of the optimal action. Such a computationally intensive task limits the applicability of such a synthesis procedure for real-time control. These issues are well understood in the control literature, particularly in the context of stabilizing control, and are typically mitigated by adopting a feedback control approach, which is inherently more robust. Time-optimal feedback controls are static maps that depend solely on the current state of the system and are evaluated as the system evolves. This adaptability allows them to correct deviations from the optimal trajectory, improving their robustness to uncertainties and disturbances.

In deterministic optimal control, the construction of an optimal feedback law is obtained via dynamic programming, leading to a static, first-order Hamilton-Jacobi-Bellman PDE for the value function of the control problem. The optimal feedback law is obtained as a by-product.  However, this method requires an accurate tracking of the discontinuous switching law, which can only be achieved using local, grid-based numerical methods, and hence it is limited to low-dimensional problems. In this section, we present an alternative synthesis method that circumvents the solution of the Hamilton-Jacobi-Bellman PDE by resorting to supervised learning. Here, we train a neural network using synthetic data from sampling open-loop, optimal trajectories obtained with the methodology presented in previous sections.  

\subsection{Feedforward Neural Network Classifier}
 Given the bang-bang (binary) nature of the control in the problem under consideration, we frame the feedback approximation as a classification task rather than relying on regression. This means that the output of our model will be of discrete nature in $\{-1,1\}$. 
 
 We consider a class of models within the family of feedforward neural networks (NN), characterized by a sequential composition of nonlinear activation functions applied component-wise to affine transformations of the inputs. 
 If we define the $m$-th layer as $l_m(\bm{z}_m)=\sigma_m\left(A_m \bm{z}_m+\bm b_m\right)$, the parametric model reads
 \begin{equation*}
 u_\theta(\bm{z}) = l_M \circ \ldots \circ l_2 \circ l_1(\bm{z})\,,
 \end{equation*}
 with trainable parameters $\{A_m,\bm b_m\}_{m=1}^{M-1}$, $A_m\in\mathbb{R}^{n_{m-1}\times n_m}$ and $\bm b_m\in\mathbb{R}^{n_m}$ being the weight matrix and bias vector of the $m$-th layer respectively, which has $n_m$ neurons. Here, the nonlinear activation function $\sigma_m$, $m=1,\cdots,M-1,$ is applied component-wise to the layer input variable $\bm{z}_m\in\mathbb{R}^{n_{m-1}}$. We fix the final layer to be sigmoid 
 \begin{equation*}
 l_M(\bm{z}_M) = \dfrac{1}{1 + e^{-\bm{z}_M}},
 \end{equation*}
 where, as before, the function is applied component-wise to the entries of the layer input $\bm{z}_M$. This function maps any real-valued input into the range $(0,1)$, making it particularly suitable for binary classification tasks, as the output can be interpreted as a probability:
\begin{equation}\label{net_prob}
\mathbb{P}\big[u(\bm{z})=1\big] = u_\theta (\bm{z}), \;\text{and\; }\mathbb{P}\big[u(\bm{z})=-1\big] =1 - u_\theta (\bm{z}).
\end{equation}
 By denoting the approximation of the feedback control as $\tilde{u}(\bm{z})\approx u(\bm{z})$, we define it as the most likely event under the probability $u_\theta$:
\begin{equation*}
\tilde{u}(\bm{z}) = 
\begin{cases} 
1, & \text{if } u_\theta(\bm{z}) \geq 0.5, \\
-1, & \text{otherwise}.
\end{cases}
\end{equation*}
This formulation not only provides a classification decision but also quantifies the confidence of the model in the forecast, in the sense that the farther \(u_\theta(\bm{z})\) is from the decision threshold \(0.5\), the higher the model’s confidence in its classification.
 
\subsection{Synthetic Data Generation}
For the generation of the training and testing datasets, we utilize the deflated Newton's method discussed in Section \ref{newt} and solve the polynomial system in \eqref{eq:polynomialSystemB0}. We begin by generating $N_s$ sample points in the state space, denoted by $\bigl\{\bm{x}^{(i)}\bigr\}_{i=1}^{N_s}$. For each initial condition, we evaluate the solvability of the polynomial system under the two possible initial control values, $u_0 = 1$ and $u_0 = -1$. 

To ensure the accurate detection of switching surfaces, we further store the state and control information for a sequence of $100$ points along the resulting optimal trajectory generated from each initial condition. This process results in a dataset consisting of $N = 100N_s$ entries:
\begin{equation*}
 \mathcal{T} = \bigl\{\bm{x}^{(i)}, u_0^{(i)}\bigr\}_{i}\cup\bigl\{\bm{y}^{(i)}(\tau_k), u^{(i)}(\tau_k)\bigr\}_{i,k}, \quad \text{for } i = 1,\dots,N_s,\; \text{and\;} k = 1,\dots,100,
\end{equation*}
where $\bm{y}^{(i)}(t)$ is the optimal trajectory departing from $\bm{x}^{(i)}$ and $\tau_k$ are uniformly spaced discrete times in $(0,T\,]$. For convenience, we will denote the input $u_\theta(\cdot)$ as $\bm{z}$, and rearrange the training set as 
\begin{equation*}
 \mathcal{T} = \bigl\{\bm{z}^{(j)}, u^{(j)}\bigr\}_{j=1}^{N}\,,
\end{equation*}
where $N=100N_s$.
\subsection{Training of the Model}
The network is trained to minimize a cross-entropy loss function on data, ensuring the output probabilities align with the provided labels. 
The training is performed using the Adam optimizer \cite{Adam}. The dataset is divided into two subsets: 90\% for training, 10\% for testing. The number of layers and neurons is selected to balance model complexity and computational efficiency, and hyperparameters such as learning rate and batch size are tuned using a grid search with the aim of maximizing the approximation performance. These choices are problem-specific and will be clarified for each numerical test in Section \ref{numtests}.

In this case, the binary cross-entropy loss function measures how well the neural network's predicted probabilities align with the true binary labels ($u(\bm{z}) \in \{-1, 1\}$). The binary cross-entropy loss for a single data point $\bm{z}$ is given by:
\begin{equation*}
L(\bm{z}, u(\bm{z})) = - \big[ \lambda \log\bigl(u_\theta(\bm{z})\bigr) + (1 - \lambda) \log\bigl(1 - u_\theta(\bm{z})\bigr) \big], 
\end{equation*}
where:
\begin{itemize}[itemsep=0.1cm,topsep=0.1cm]
 \renewcommand{\labelitemi}{\tiny{$\blacksquare$}}
 \item $\lambda \in \{0, 1\}$ is the true label, with $\lambda = 0$ corresponding to $u(\bm{z}) = -1$ and $\lambda = 1$ corresponding to $u(\bm{z}) = 1$,
 \item $u_\theta(\bm{z})$ is the network's predicted likelihood that $u(\bm{z})=1$, as in \eqref{net_prob}.
\end{itemize}

For a dataset of $N$ samples, the total binary cross-entropy loss is:
\begin{equation}\label{eq:loss}
 \mathcal{L} = \frac{1}{N} \sum_{j=1}^N \bigl[ -\lambda_j \log\bigl(u_\theta\bigl(\bm{z}^{(j)}\bigr)\bigr) - (1 - \lambda_j) \log\bigl(1 - u_\theta\bigl(\bm{z}^{(j)}\bigr)\bigr) \bigr].
\end{equation}
This loss function encourages the network to predict probabilities $u_\theta(\bm{z}^{(j)})$ close to 1 when $\lambda_j = 1$ and close to 0 when $\lambda_j = 0$, effectively learning the mapping from states $\bm{z}^{(j)}$ to control labels $u^{(j)}$.

The model's performance is assessed by measuring the proportion of correctly predicted labels out of the total number of samples 
\begin{equation*}
\text{Accuracy} = \frac{1}{N} \sum_{j=1}^N \bm{1}_{\bigl(\tilde{u}(\bm{z}^{(j)})=u^{(j)}\bigr)},
\end{equation*}
where $\bm{1}$ denotes the indicator function.


\section{Numerical Tests}\label{numtests}
In this section, we evaluate the proposed methodology on an {$n$-th} order integrator, where $n$ ranges from $2$ to $5$. The discussion includes the process of data generation, the architecture of the model, and the choice of hyperparameters. In particular, the network architecture was selected via a grid search. We explored NNs with between \(1\) and \(3\) hidden layers, and number of neurons drawn from the set \(\{10,20,30,40,50,60,80,100\}\). Across this search space, the \(tanh\) activation function consistently yielded the most stable and accurate results. For completeness, we also evaluated $ReLU$ and $sigmoid$ activations, but these did not improve performance in our setting. The final architectures reported in the following tests correspond to the best-performing models identified from the above grid.  

We compare the trajectories controlled by the trained model against the optimal trajectories obtained using Algorithm \ref{alg:deflation} to solve \eqref{eq:polynomialSystemB0}. Across the numerical examples, we demonstrate that interpreting $u_\theta$ as a binary classifier with confidence information improves its approximation performance. 
We evaluate the robustness of the feedback approximation under noisy trajectories and highlight the computational advantages of using Hermite quadratic forms to determine a bound for the number of roots to be identified through deflation.
\subsection{The double integrator.}
We consider the time-optimal control problem 
\begin{equation*}
 \min\limits_{u(t)} T\qquad \text{such that }\quad\begin{cases}
 \dot{y}_1 = y_2\,,\qquad y_1(0) = x_1\,,\qquad {y}_1(T) = 0\,,\\
 \dot{y}_2 = u\,,\qquad\; y_2(0) = x_2\,,\qquad {y}_2(T) = 0\,,\\
 \end{cases}
\end{equation*}
whose solution is given by a piecewise constant control, starting from an initial control $u_0 \in \{-1,1\}$ and a switching sequence $\bm{t}= (t_1,t_2)^\top$, computed by solving the associated polynomial system \eqref{eq:polynomialSystemB0}, which for $n=2$ reads:
\begin{equation}\label{eq:2dtest}
\begin{aligned}
 \begin{cases}
 t_1 - t_2 + x_{2} = 0\\
\dfrac{t_1^2}{2} + t_1 t_2 - \dfrac{t_2^2}{2} + x_{1} + x_{2}(t_1+t_2) = 0
 \end{cases}\qquad \text{for }\; &(u_0=1),\\
 \begin{cases}
 t_1 - t_2 - x_{2} = 0\\
\dfrac{t_1^2}{2} + t_1 t_2 - \dfrac{t_2^2}{2} - x_{1} - x_{2}(t_1+t_2) = 0
 \end{cases}\qquad \text{for }\; &(u_0=-1).
 \end{aligned}
\end{equation}
Admissible solutions of \eqref{eq:2dtest} specify the optimal trajectory, starting from $\bm{x}=(x_1,x_2)^{\top}$ with control $u_0$ and switching sequence\footnote{For the sake of clarity, we recall that this switching sequence is to be interpreted in the following way: there is an initial control value $u_0$ in $[0,t_1)$, which then switches to $-u_0$ in $[t_1,t_1+t_2]$, being the minimum time $T=t_1+t_2$.} $\bm{t}$. We generate a dataset by sampling $N_s = 50$ uniformly distributed initial conditions $\bigl\{\bm{x}^{(i)}\bigr\}_{i=1}^{N_s}$ in the numerical domain $[-1,1]^2$, generating couples of optimal state-actions $\bigl\{\bm{y}^{(i)}(\tau_k),u^{(i)}(\tau_k)\bigr\}$ for $0 = \tau_0 < \cdots \tau_k < \cdots \tau_{100}= T $. 

We rely on the $N=100\cdot N_s$ generated data pairs
\begin{equation*}
 \mathcal{T} = \{\bm{z}^{(j)},u^{j}\}_{j=1}^{N},
\end{equation*}
to train a NN classifier mapping states $\bm{z}$ into the associated feedback $u_\theta(\bm{z})$. After a grid search over a number of possible feedforward architectures, we select $u_\theta(\cdot)$ to have a single hidden layer of width $n_2 = 100$ neurons and $\sigma_2=\tanh(\cdot)$ as activation function. 
We train the model to minimize the loss function \eqref{eq:loss} over the training set via Adam. 

After training, the model achieves a test set accuracy of 
 $99.38\%$ with a corresponding loss $\mathcal{L} = 0.0156$.
In Figure \ref{fig:classifier2d}, we show how the trained classifier labels points in $[-0.25,0.25]^2$. The model assigns controls with high probability to points on either side of the switching surface, while points lying directly on the surface are associated with lower confidence, as indicated by the lighter colors in the figure.
\begin{figure}
 \centering
 \includegraphics[width=0.5\textwidth]{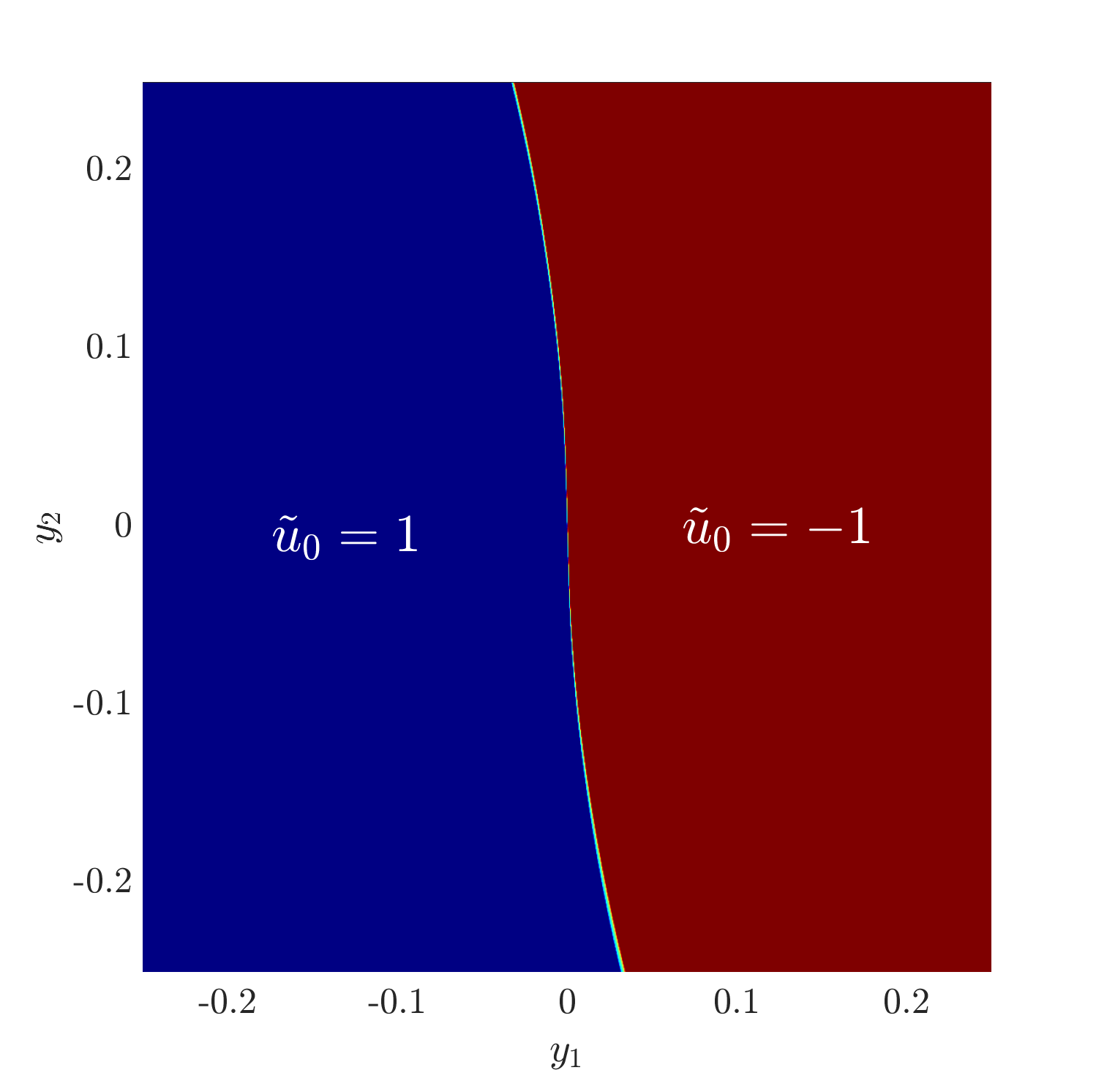}
 \caption{The trained classifier identifies the switching surface as a low-confidence region.}
 \label{fig:classifier2d}
\end{figure}

All optimal trajectories eventually intersect the switching hyperplane and travel along it until reaching the origin. Therefore, the model's accuracy on the switching surface is crucial for achieving a reliable feedback approximation. At this point, we have two options: 
\begin{itemize}[itemsep=0.1cm,topsep=0.1cm]
 \renewcommand{\labelitemi}{\tiny{$\blacksquare$}}
 \item we rely on the classification provided by the neural network $u_\theta$, disregarding its low confidence,
 \item or we define a confidence threshold, and if the confidence falls below this threshold, we invoke the open-loop solver to determine the feedback control at the low-confidence point. 
\end{itemize}
We compare these two procedures in Figure \ref{fig:paths2d}. The trajectory controlled through the neural feedback approximation is integrated via explicit Euler, with time-step $5^{-4}$. The enhanced approximation---which is based on the polynomial system solution when the classifier confidence is below $\varepsilon = 0.01$---is triggered on $1\%$ of the trajectory points and does not improve significantly the approximated optimal horizon $T_{NN}$. We conclude that, in this first example, even low-confidence points are classified correctly via $u_\theta$.
\begin{figure}[ht!]
 \centering
 \includegraphics[width=0.45\textwidth]{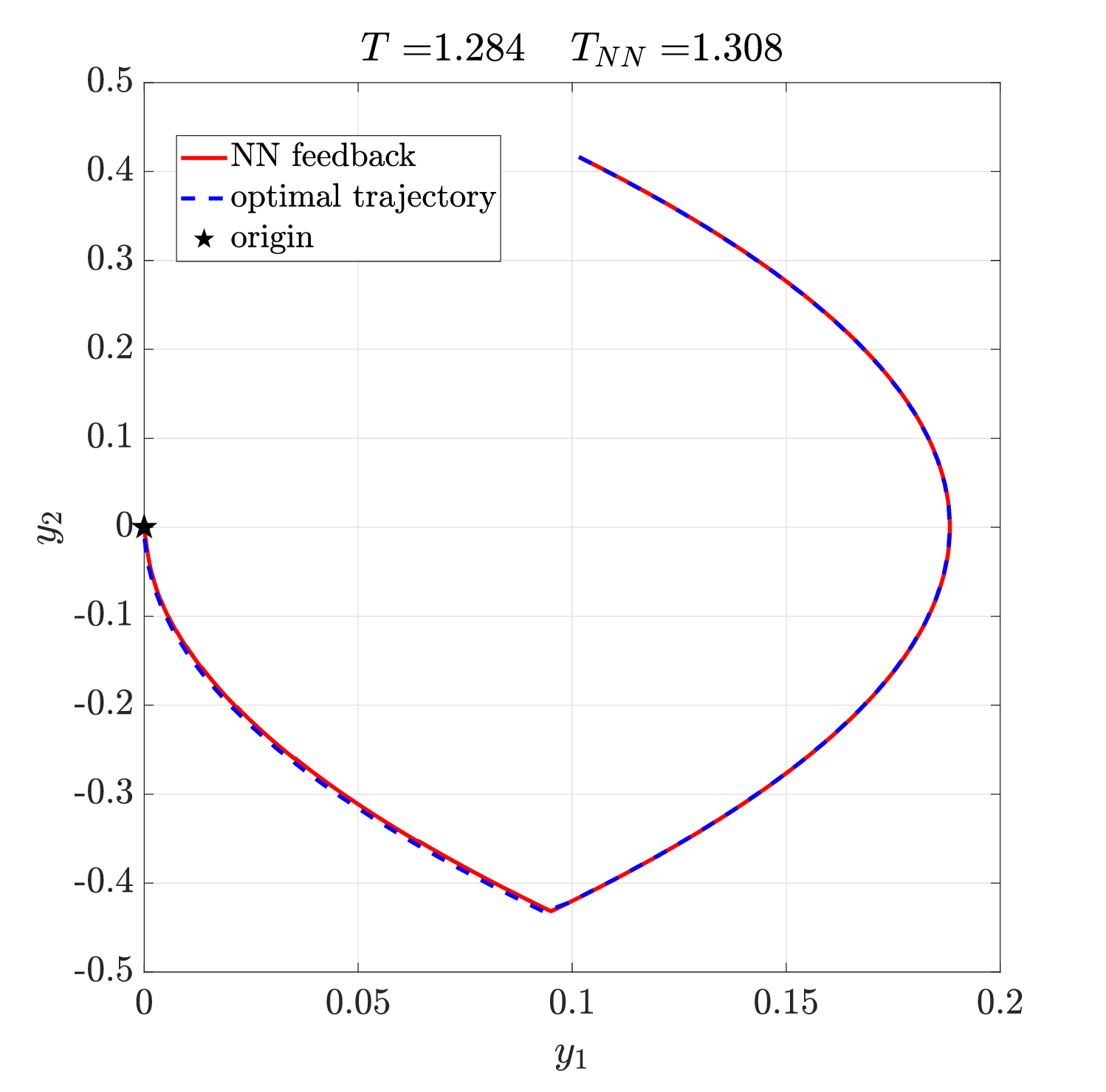}
 \includegraphics[width=0.45\textwidth]{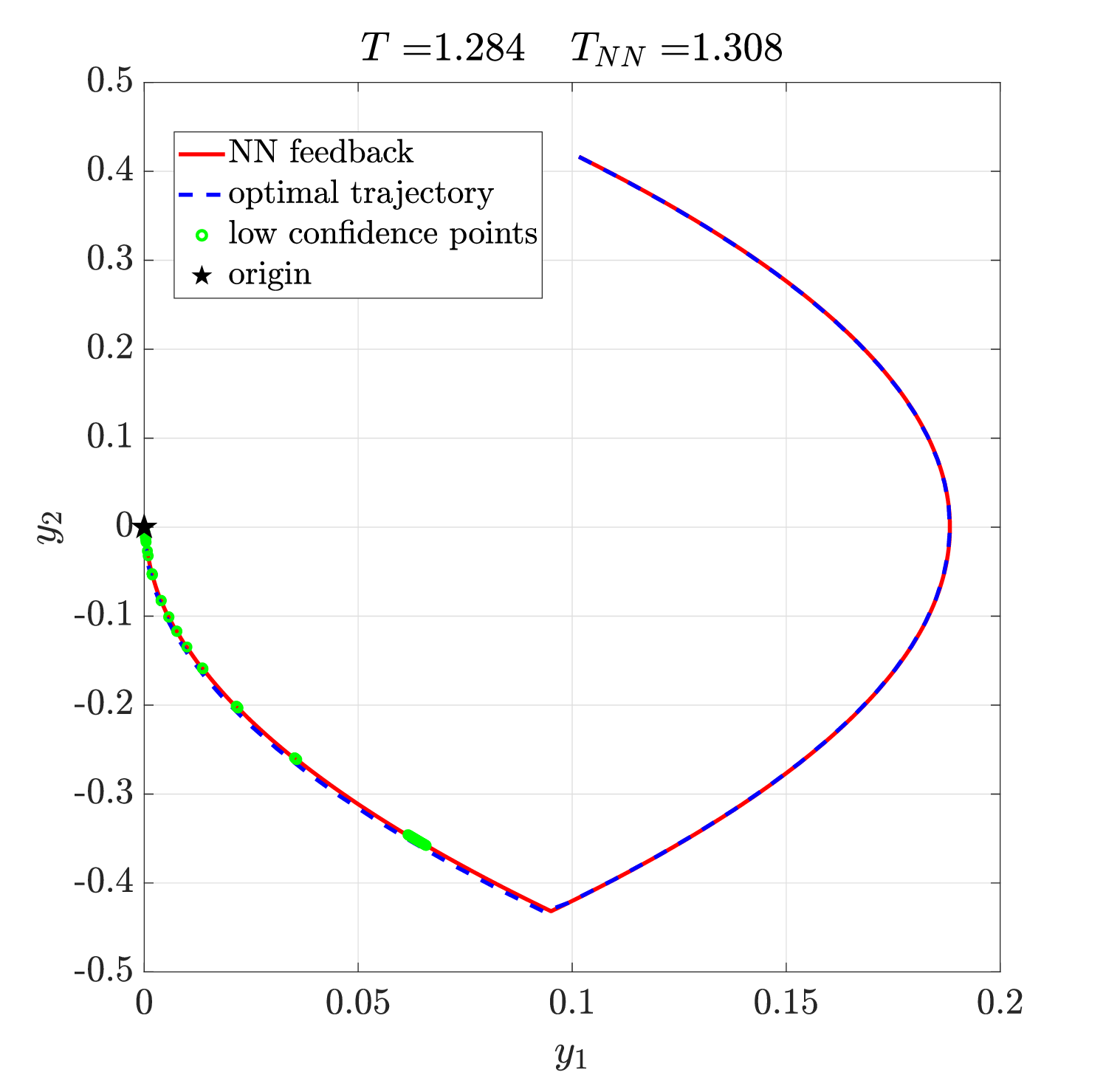}
 \caption{Comparison of the optimal trajectory computed via deflation with the one controlled with the classifier $u_\theta$. On the left, the NN feedback is solely determined via $u_\theta$, whilst on the right we rely on the identification of the solution of the polynomial system \eqref{eq:2dtest} to identify the feedback at low-confidence points. $T$ denotes the exact minimum time.}
 \label{fig:paths2d}
\end{figure}

\subsection{Triple integrator}
Similarly as in the first test, we consider the time-optimal control problem subject to 
\begin{equation*}
 \begin{cases}
 \dot{y}_1 = y_2\\ 
 \dot{y}_2 = y_3\\ 
 \dot{y}_3 = u
 \end{cases}
\end{equation*}
with state variable $\bm{y}=(y_1,y_2,y_3)^{\top}$ and boundary conditions $\bm{y}(0)= \bm{x}=(x_1,x_2,x_3)^{\top}$, $\bm{y}(T)=\bm{0}$. The polynomial system \eqref{eq:polynomialSystemB0} features $3$ equations for the unknown switching sequence $\bm{t}=(t_1,t_2,t_3)$. We sample state and controls along the optimal trajectories departing from $N_s = 5000$ uniformly distributed initial conditions in $[-1,1]^3$. The resulting dataset is used to train a model $u_\theta$ characterized by a single hidden layer with width $n_2 = 80$ and $\sigma_2 = \tanh(\cdot)$ activation function. After training, the model reaches $99.12\%$ accuracy in the test set, associated with a loss $\mathcal{L} = 0.0334$.
In this example, the enhanced feedback approximation, calling the polynomial solver for identifying the control at points with confidence below $\varepsilon = 0.005$, occurs on $4.08\%$ of the trajectory points, and improves the approximated horizon when compared to the trajectory controlled solely via $ u_\theta$. We show the results in Figure \ref{fig:paths3d}, where the NN-controlled trajectory is computed via a forward Euler scheme with stepsize $10^{-3}$, whilst the optimal trajectory is computed through exact integration.
 \begin{figure}[h!]
 \centering
 \includegraphics[width=0.45\textwidth]{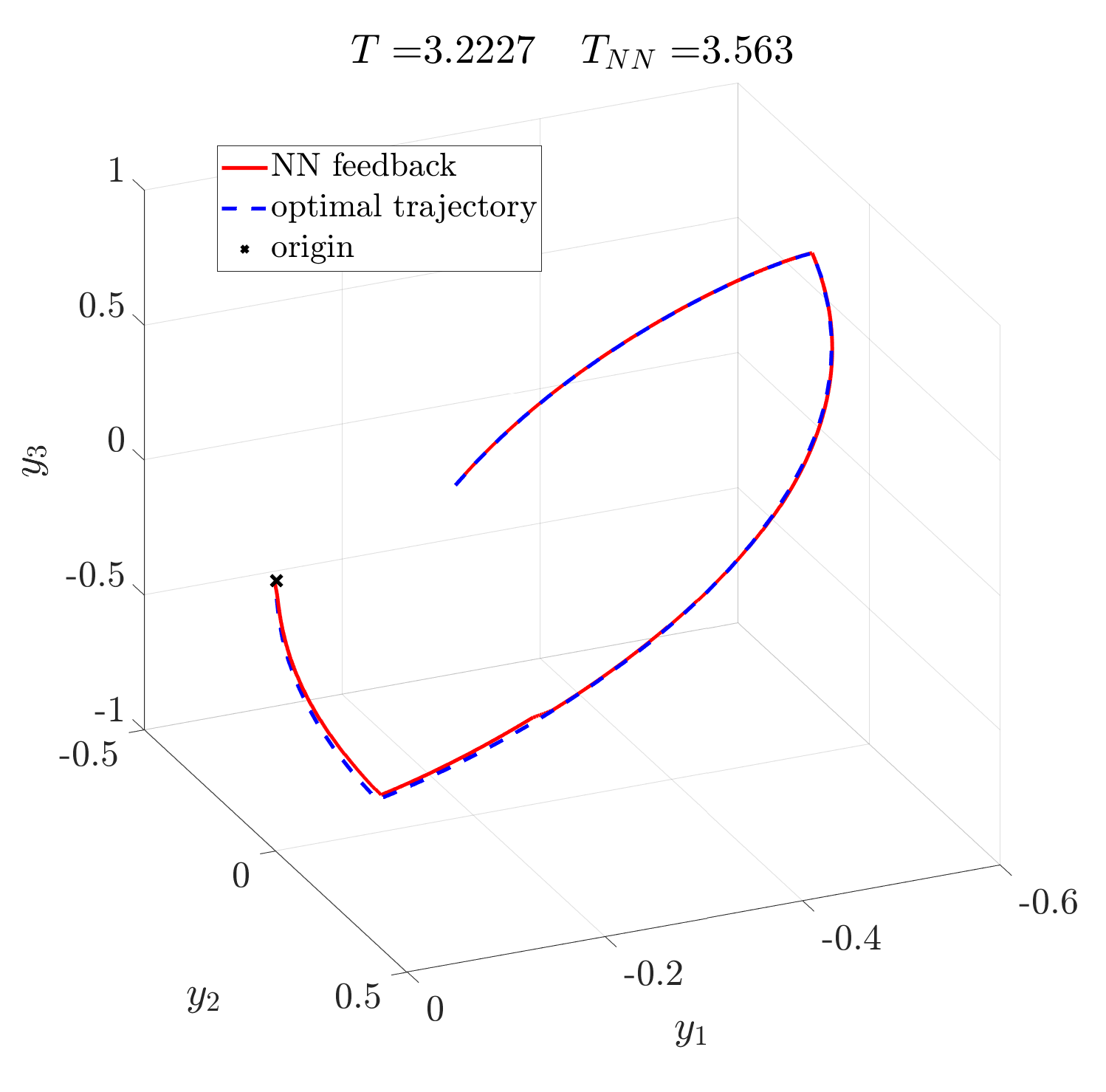}
 \includegraphics[width=0.45\textwidth]{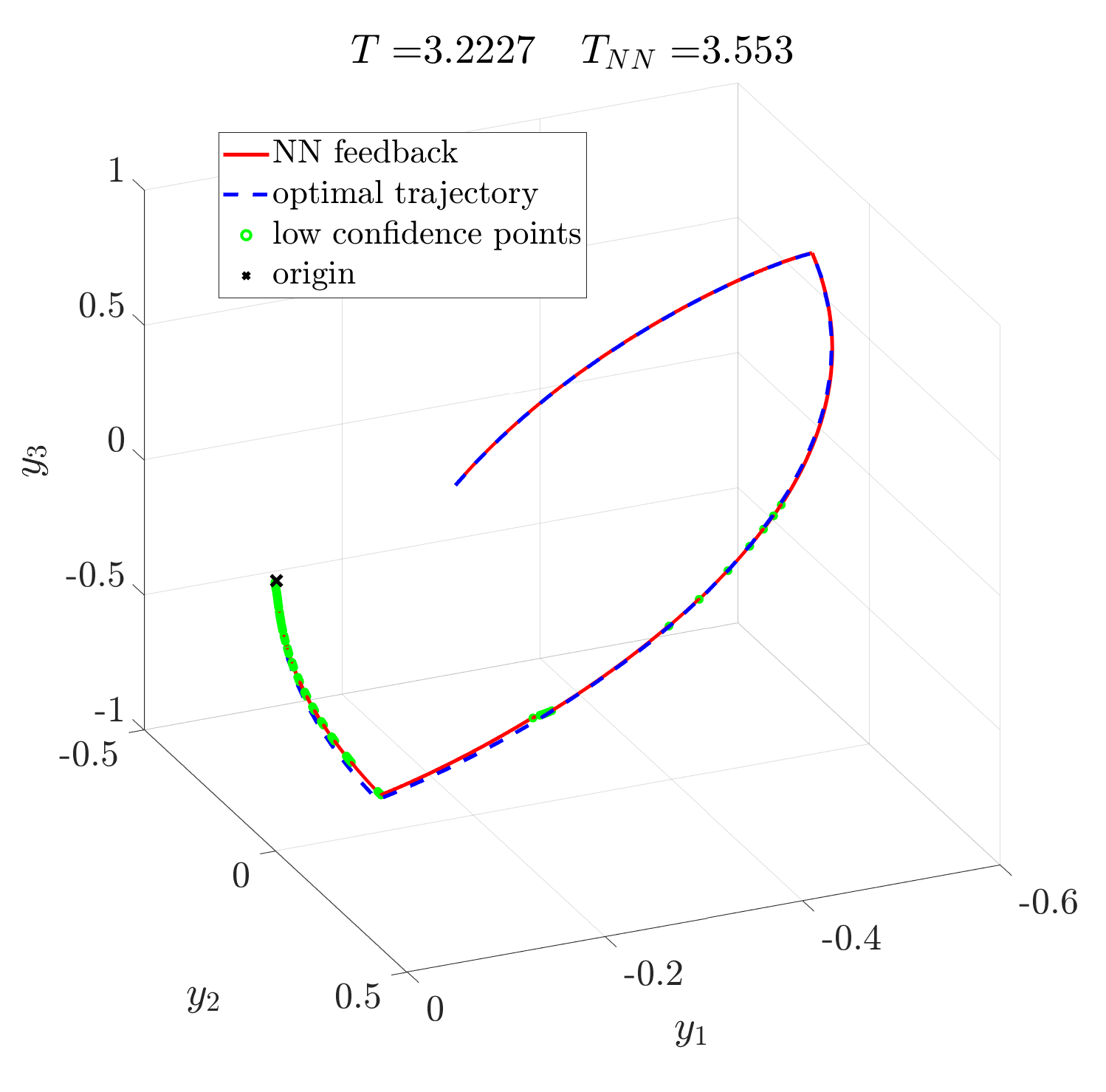}
 \caption{Comparison of the optimal trajectory with the $u_\theta$-controlled (left) and with the confidence-enhanced approximation (right). Relying on the polynomial solver for low-confidence points improves the approximation, as the time $T_{NN}$ needed to reach the origin decreases.}
 \label{fig:paths3d}
\end{figure} 

We motivated the introduction of the feedback NN approximation as a way to achieve robustness against noise. 
To evaluate this, we introduce white noise into the last component of the system. Specifically, at every discrete integration time $t$ we add $\bigl(0,0,\zeta(t)\bigr)$, where $\zeta(t)\sim\mathcal{N}(0,\sigma^2)$, with $\sigma^2=0.02$. Starting from the same initial condition, we perform a Monte Carlo simulation for both the open-loop and the approximated feedback control approaches. 
Figure \ref{fig:robustness3d} illustrates the mean and variance obtained from $1000$ noisy controlled trajectories. 
As expected, while the open-loop solution diverges on average from the target final state, the approximated feedback law successfully drives the system towards the origin.

\begin{figure}[ht!]
 \centering
 \includegraphics[width=0.6\textwidth]{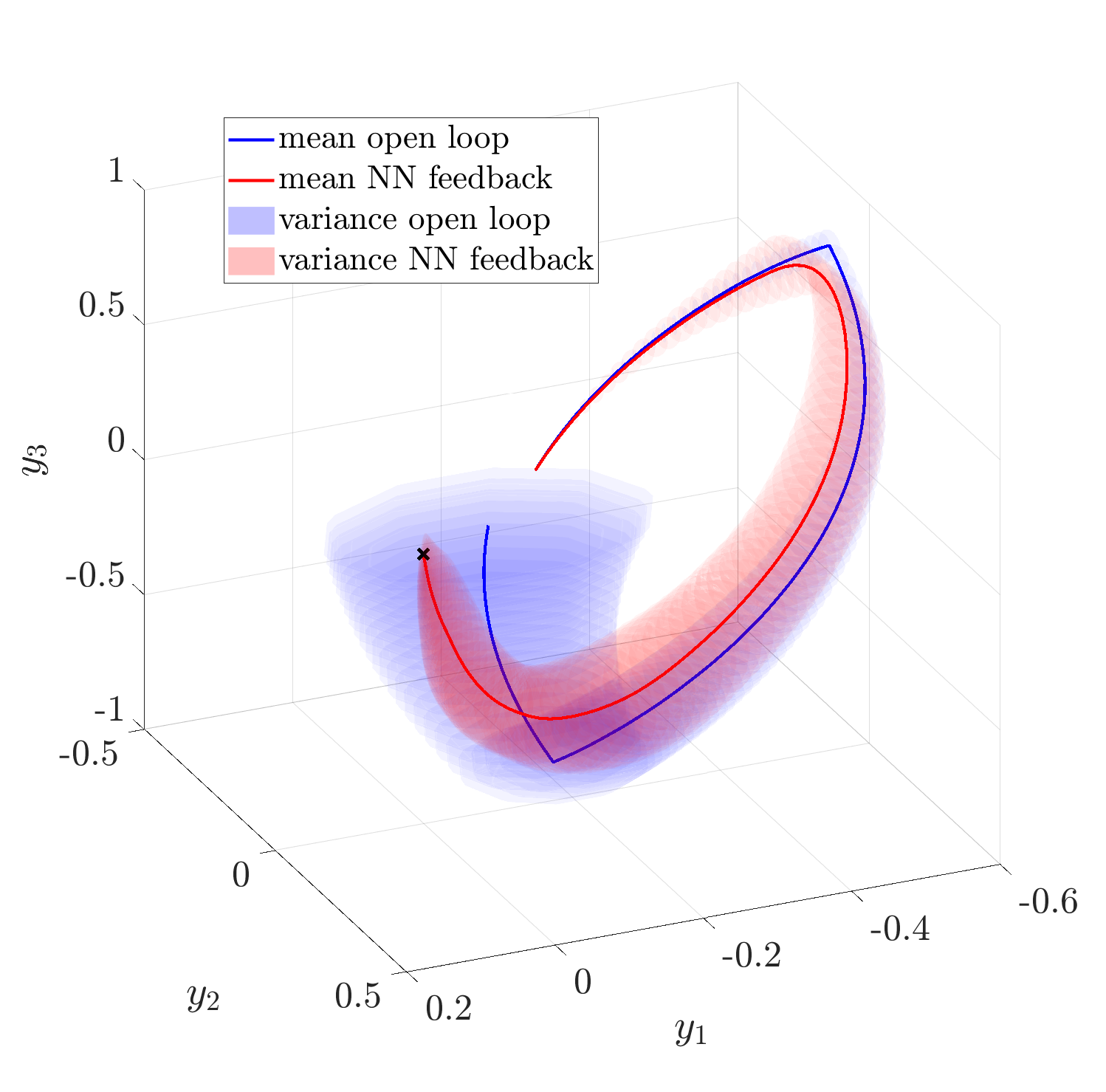}
 \caption{Mean and variance of a Monte Carlo simulation with $1000$ trajectories perturbed by Gaussian noise. The robustness of the approximated feedback law ensures that the controlled system reaches the target destination, while the open-loop solution, obtained as the solution of the polynomial system \eqref{eq:polynomialSystemB0}, diverges from it.}
 \label{fig:robustness3d}
\end{figure}

\subsection{4th-order tests}
We consider the time-optimal control problem subject to the fourth-order integrator for state variable $\bm{y}=(y_1,\dots,y_4)^{\top}$ and boundary conditions $\bm{y}(0)=\bm{x}=(x_1,\dots,x_4)^{\top}$, $\bm{y}(T)=\bm{0}$. In this higher-dimensional example, we examine the upper bound on the number of solutions of the associated system \eqref{eq:polynomialSystemB0} for the switching sequence $t_1,\dots,t_4$. The theoretical bound on the total number of real and complex solutions is $4!=24$. However, as discussed in Section \ref{hermy}, the use of Hermite quadratic forms allows us to refine this bound by identifying the number of real solutions more precisely. 

We sample initial conditions $\bm{x}^{(i)}\in\bigl([-1,1]\cap\mathbb{Q}\bigr)^4$, for $i=1,\dots,N_s = 10^4$, and for each $\bm{x}^{(i)}$ we compute the number of real roots of the polynomial system associated with $u_0 = 1$ and $u_0=-1$, by computing the Hermite quadratic form, denoted by $\mathcal{H}(J^+)$ and $\mathcal{H}(J^-)$, respectively. 
Note the transition from $\RR$ to $\QQ$, which we do as computation of the Hermite quadratic form over the reals is computationally unreliable, if not impossible. 
This preliminary step reduces the computational cost of data generation, as the bound for the deflation algorithm \ref{alg:deflation} is tightened relative to the theoretical one. The range of real roots given by the Hermite quadratic form varies from 0 to 8, depending on the initial condition. For $\bm{x}^{(i)}$'s with $0$ real solutions for $J^+$ or $J^-$, this directly determines the sign of the initial control $u_0$. These new bounds reduce the CPU time for data generation from $108$ to $49$ minutes when compared to the theoretical bound. 

\begin{figure}[ht!]
 \centering
 \includegraphics[width=0.5\textwidth]{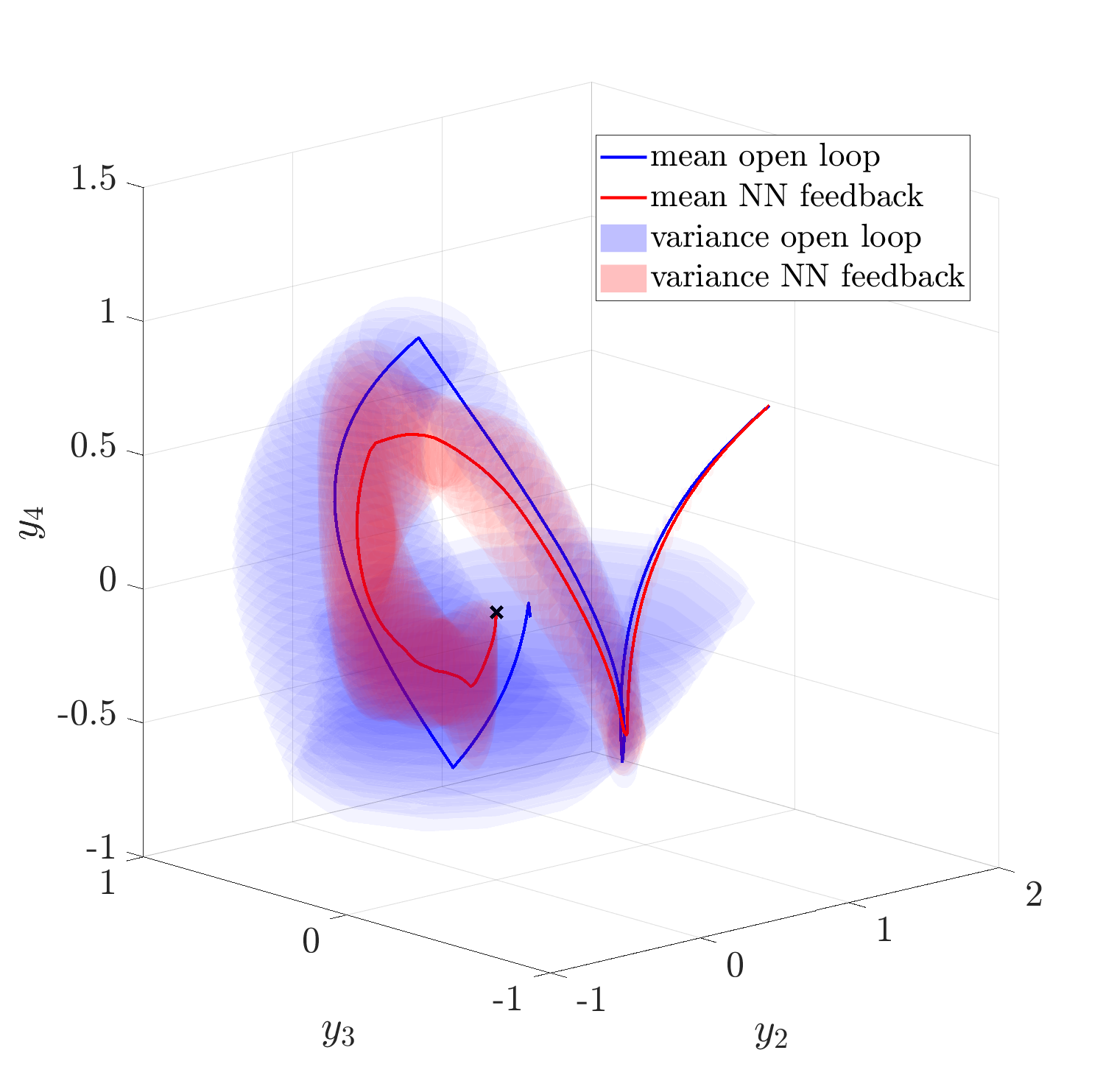}\\
 \includegraphics[width=0.4\textwidth]{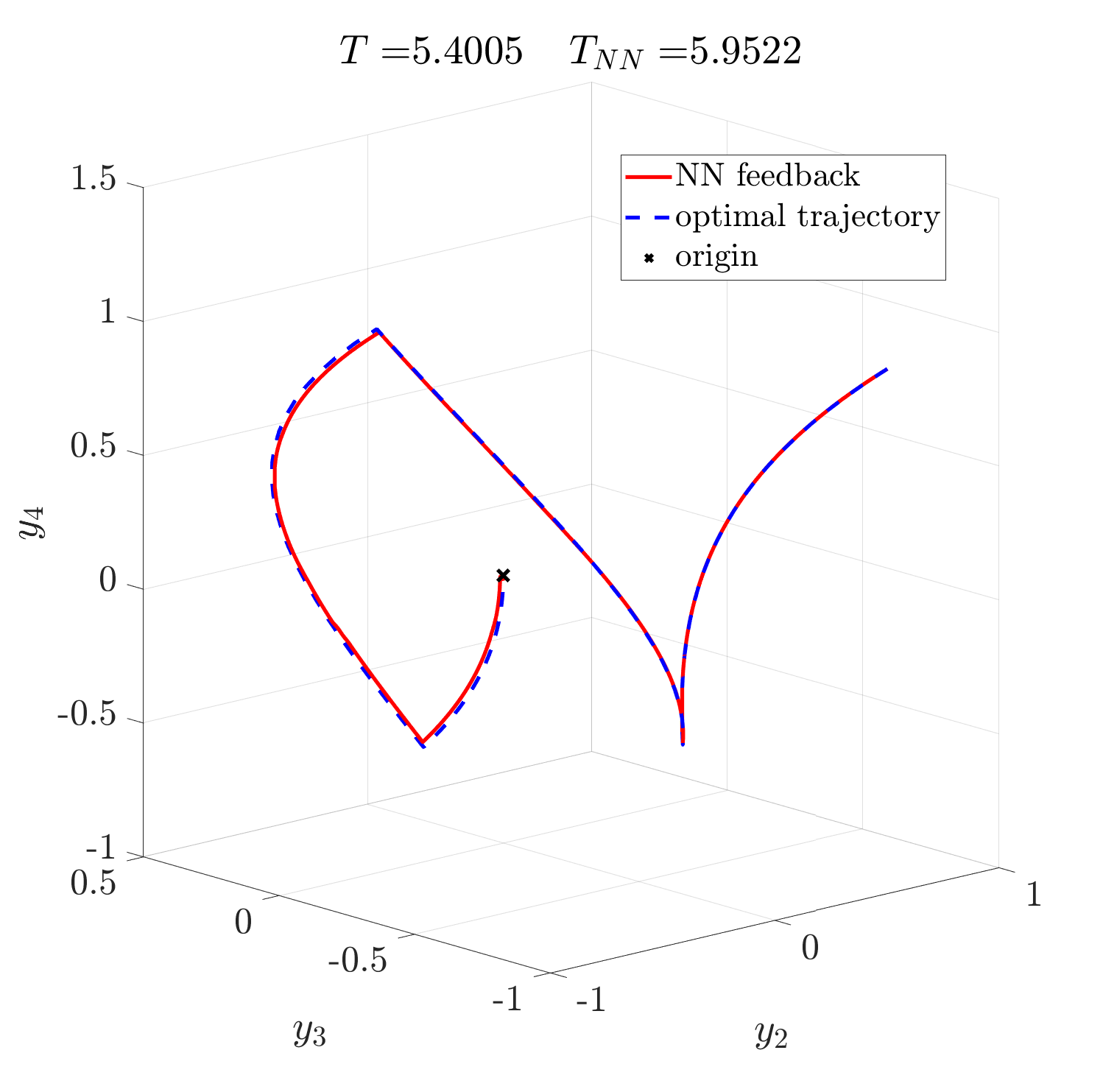}
 \includegraphics[width=0.4\textwidth]{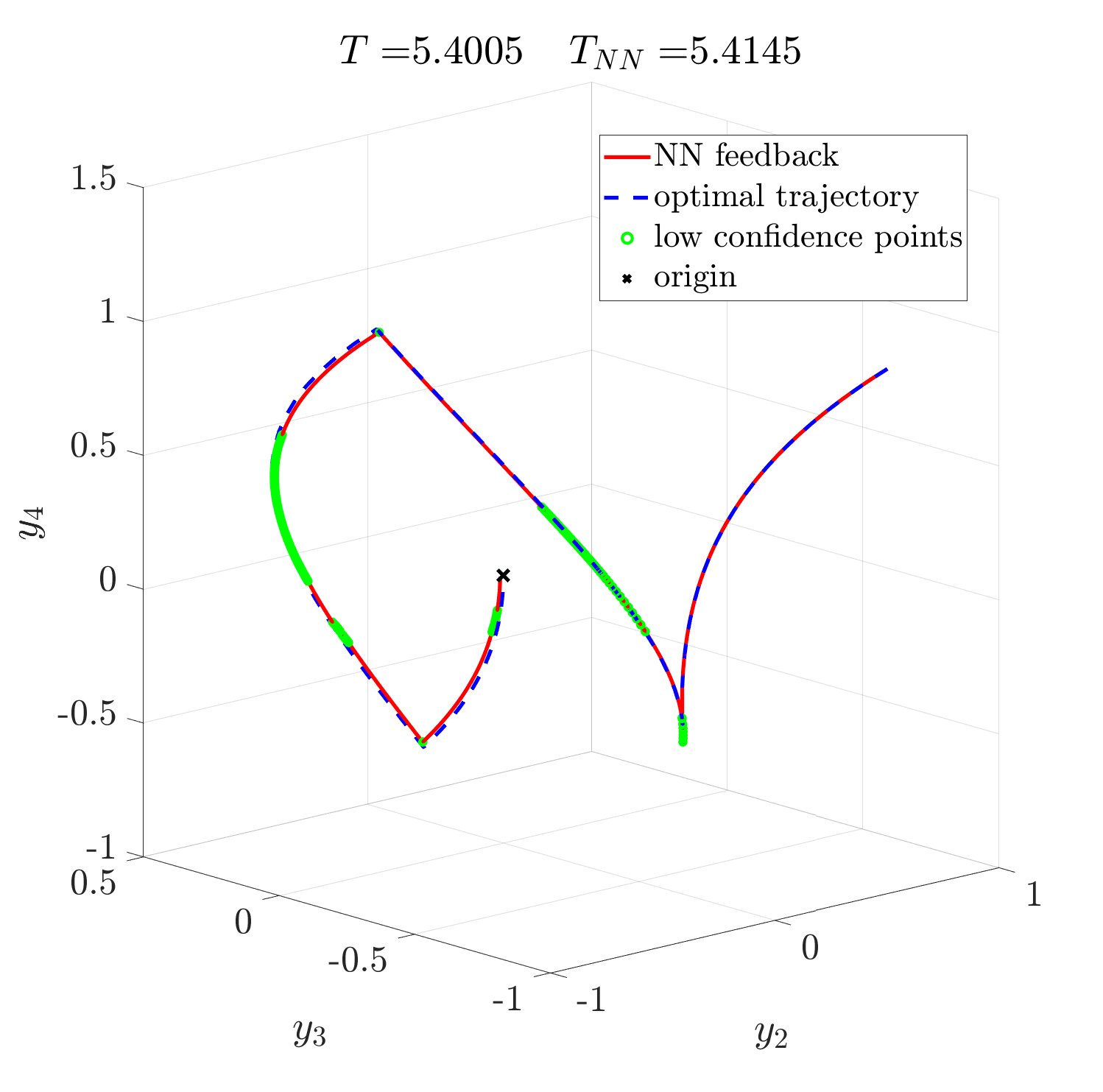}
 \caption{Monte Carlo simulation of $1000$ controlled noisy trajectories via open loop control signal vs. approximated feedback control (top). Comparison of the optimal trajectory with the one controlled via the approximated feedback law (left). Calling the polynomial solver for the low-confidence points improves the approximation performance in terms of time horizon $T_{NN}$ (right).}
 \label{fig:paths4d}
\end{figure}
We train a model $u_\theta$ characterized by $2$ hidden layers, with width $n_2=n_3 =100$ and $\sigma_2=\sigma_3=\tanh(\cdot)$ activation functions. 
After training, we evaluate the model in the test set, obtaining the accuracy of $96.76\%$ and the loss $\mathcal{L} = 0.0882$. 
In Figure \ref{fig:paths4d}, we analyze the controlled trajectories comparing the optimal solution against its approximation along the dimensions $y_2,y_3,y_4$. We consider a low-confidence threshold $\varepsilon=0.005$, under which we rely on the deflation routine to determine the control. 
Low-confidence classification triggers on $1.3\%$ of the trajectory points. 
Even more prominently than in the lower order examples, the enhanced approximate feedback improves the performance of the control law. As in the previous example, we also test the robustness of the feedback approximation by injecting additive noise of the form $\bigl(0,0,0,\zeta(t)\bigr)$, with $\zeta(t)\sim\mathcal{N}(0,\sigma^2)$, where $\sigma=2^{-2}$. A Monte Carlo simulation of $1000$ controlled noisy trajectories shows how the approximated feedback control successfully steers the system towards the origin, while the open-loop solution diverges. 

\subsection{5th-order integrator}
Finally, we consider the {$5$-th} order integrator for state $\bm{y}=(y_1,\dots,y_5)^{\top}$ with boundary conditions $\bm{y}(0)=\bm{x}=(x_1,\dots,x_5)^{\top}$ and $\bm{y}(T) = \bm{0}$. The theoretical upper bound for the number of solutions $(t_1, \dots, t_5)$ of the associated polynomial system is $5! = 120$. However, excessively deflating a polynomial system can lead to instability. Heuristically, as the number of identified and deflated roots increases, the root-finding procedure may diverge, failing to locate additional existing solutions. 
This pathological behaviour naturally worsens as the upper bound increases. While still feasible, the CPU time associated to the computation of the Hermite quadratic form can be used instead to sample different initial conditions with the un-informed deflation algorithm.

We sample $N_s = 50,000$ points $\bm{x}^{(i)}\in[-1,1]^5$, for $i = 1, \dots, N_s$, and include in the dataset $\mathcal{T}$ only the optimal trajectories corresponding to initial conditions for which the deflation algorithm successfully identifies an admissible solution. This selection process excludes approximately $20\%$ of the initial conditions. We consider a model $u_\theta$ with $2$ hidden layers having width $n_2=n_3=80$ and $\sigma_2=\sigma_3 = \tanh(\cdot)$ activation functions. The trained classifier achieves accuracy {$99.61\%$} in the test set, associated with loss $\mathcal{L} = 0.0141$. In Figure \ref{fig:paths5}, we test the performance of the NN feedback controller against the optimal trajectory. The performance of the approximation is significantly improved by selectively intervening with the polynomial system solver whenever the classifier confidence drops below $\varepsilon=0.1$. This occurs only at $5$ low-confidence points, all located at the intersection of switching surfaces.
\begin{figure}[h!]
 \centering
 \includegraphics[width=0.45\textwidth]{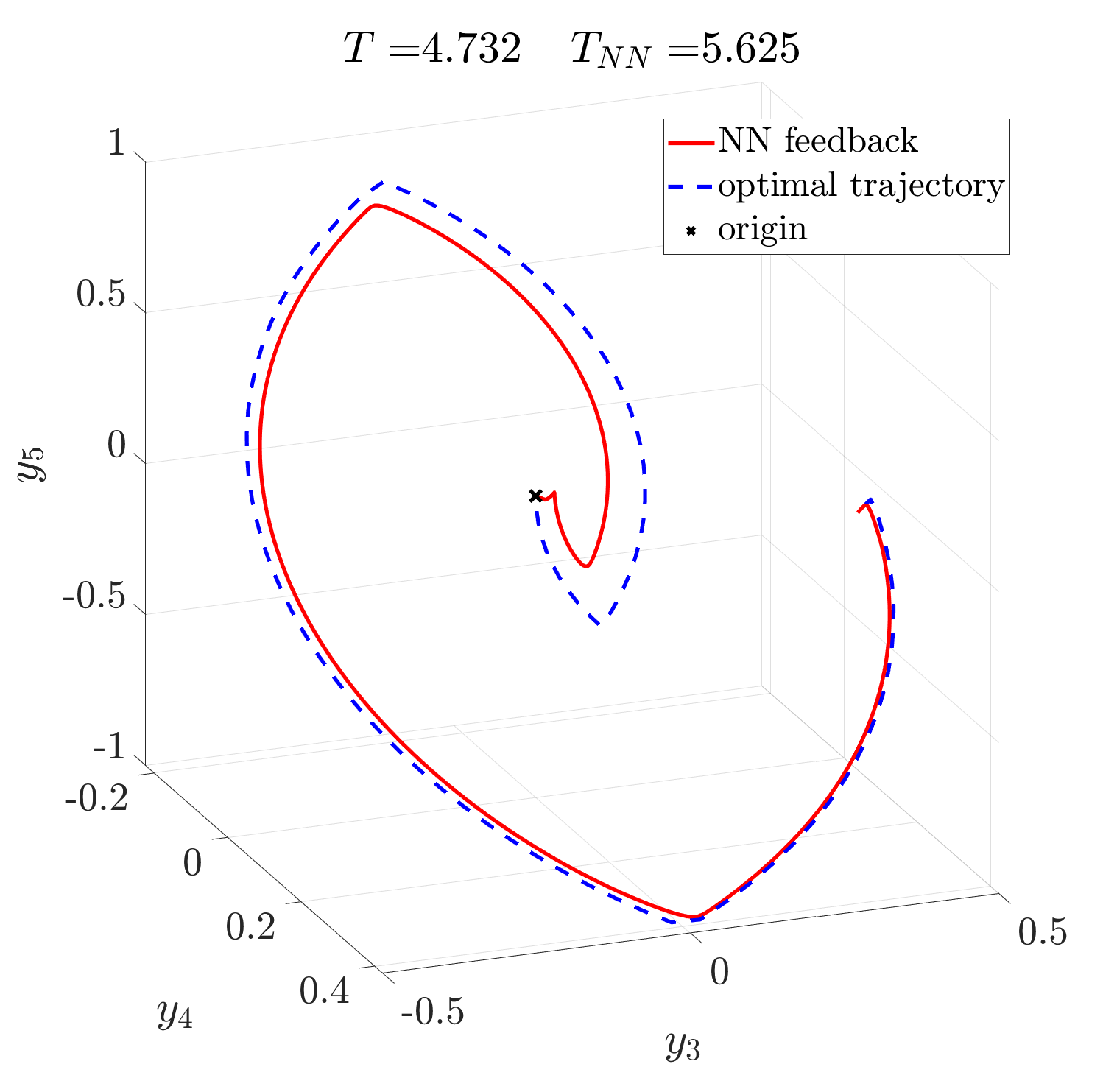}
 \includegraphics[width=0.45\textwidth]{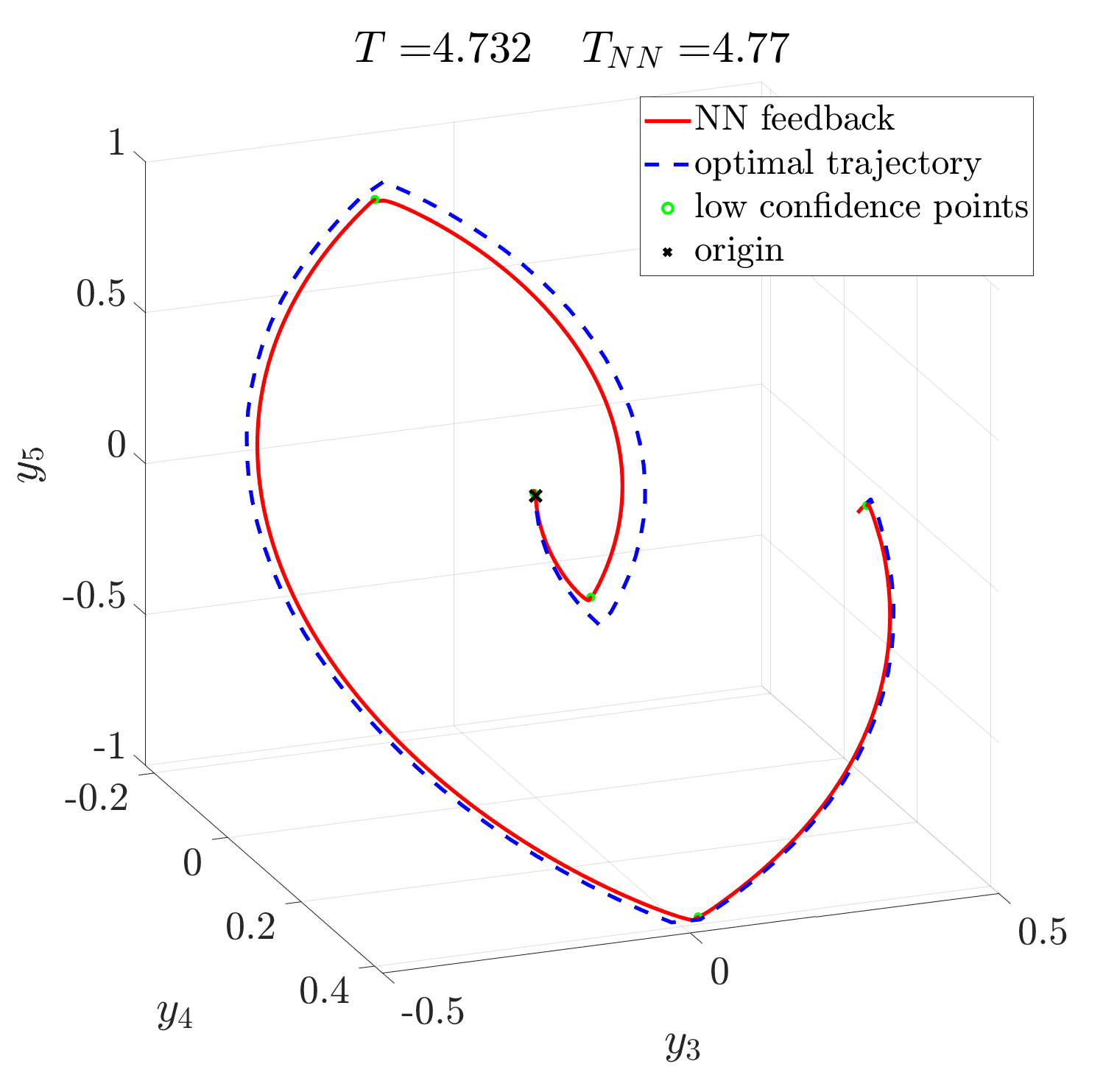}
 \caption{Comparison of the optimal trajectory with the trajectory controlled using the approximated feedback law (left). Leveraging the polynomial solver for low-confidence points enhances the approximation's performance in terms of the time horizon $T_{NN}$ (right).}
 \label{fig:paths5}
\end{figure}
\section{Concluding remarks}
In this paper, we have revisited a classical time-optimal control problem from a perspective combining polynomial systems, numerical analysis, and statistical machine learning. Our original motivation was to explore the scalability of the approach originally developed in \cite{nilpotentFeedback}, where the use of Gr\"obner bases and techniques from real algebraic geometry were proposed to solve polynomial systems arising in time-optimal control. In this respect, we have reached a negative result: to the best of our experience, the computational cost and complexity of such tools become prohibitively expensive as the dimension of the state space increases beyond 3, rendering them inapplicable for constructing real-time feedback controllers. Hence, we have resorted to the numerical solution of polynomial systems using Newton’s method on a deflated variant, which can effectively exhaust all possible roots. Here, the Hermite quadratic form is proposed as a tool to establish a bound on the number of roots of the polynomial system, increasing the efficiency of the deflation algorithm. Overall, this constitutes a consistent computational workflow for the solution of time-optimal trajectories for nilpotent linear systems.

In the second part of the paper, we adopt a statistical machine learning approach to interpret the time-optimal feedback controller as a binary classifier over the state space. This allows us to construct a control law via supervised learning, yielding a bang-bang control with a confidence indicator. This is particularly relevant to develop a robust, chattering-free control in the vicinity of the switching surface.  Our numerical tests show that, enhancing the feedback law with an open-loop solver whenever the confidence of the classifier is low, leads to a fair trade-off between accuracy of the approximate optimal trajectory, robustness, and computational cost. 

This work suggests different avenues for future research. The use of the Hermite quadratic form to determine the number of roots of time-optimal polynomial systems is promising, and merits a deeper consideration. The interpolation of the bang-bang feedback law as a binary classifier enables the application of machine learning methods for a class of control problems that remains elusive to other type of global approximation techniques due to the non-smoothness of the optimal control field. Finally, the proposed methodology can be applied to other time-optimal and sparse control problems with bang-bang or bang-zero-bang structures. 

\paragraph{\bf Acknowledgments.} We thank two anonymous reviewers for their careful reading of the manuscript and for comments that helped improve the clarity and presentation of the paper. This material is based upon work supported by the Air Force Office of Scientific Research under award number FA8655-26-1-B011.

\paragraph{\bf Conflicts of interest.} On behalf of all authors, the corresponding author states that there is no conflict of interest.

\bibliographystyle{abbrv}
\bibliography{citations}

\end{document}